\documentclass[11pt]{article}

\usepackage{amsmath}
\usepackage{amsfonts}
\usepackage{amsthm}
\usepackage{amssymb}
\usepackage[english]{babel}
\usepackage{graphicx}
\usepackage[all]{xy}

\usepackage[active]{srcltx}

\numberwithin{equation}{section}

\newtheorem{theorem}{Theorem}[section]
\newtheorem{lemma}[theorem]{Lemma}
\newtheorem{definition}[theorem]{Definition}
\newtheorem{corollary}[theorem]{Corollary}
\newtheorem{proposition}[theorem]{Proposition}
\newtheorem{remark}[theorem]{Remark}
\newtheorem{example}[theorem]{Example}

\newcommand{\Hom}{{\mathrm{Hom}}}
\newcommand{\Ext}{{\mathrm{Ext}}}
\newcommand{\Rad}{{\mathrm{Rad}}}
\newcommand{\Top}{{\mathrm{Top}}}
\newcommand{\Soc}{{\mathrm{Soc}}}

\newcommand{\res}{{\rm Res} }
\newcommand{\Res}{{\rm Res}^{\fg}_{\fg_{\oa}} }
\newcommand{\ind}{{\rm Ind} }
\newcommand{\Ind}{{\rm Ind}^{\fg}_{\fg_{\oa}} }
\newcommand{\Ann}{{\rm Ann} }
\newcommand{\Spec}{{\rm Spec} }
\newcommand{\ch}{{\rm ch} }

\newcommand{\ol}{\overline}
\newcommand{\tto}{\twoheadrightarrow}

\newcommand{\mC}{\mathbb{C}}
\newcommand{\mN}{\mathbb{N}}

\newcommand{\mZ}{\mathbb{Z}}

\newcommand{\mP}{\mathbb{P}}

\newcommand{\mg}{\mathfrak{g}}
\newcommand{\mm}{\mathfrak{m}}

\newcommand{\fa}{{\mathfrak a}}
\newcommand{\fb}{{\mathfrak b}}

\newcommand{\fg}{{\mathfrak g}}
\newcommand{\fh}{{\mathfrak h}}

\newcommand{\fl}{{\mathfrak l}}
\newcommand{\fn}{{\mathfrak n}}

\newcommand{\fp}{{\mathfrak p}}

\newcommand{\fu}{{\mathfrak u}}

\newcommand{\cD}{\mathcal{D}}

\newcommand{\cP}{\mathcal{P}}

\newcommand{\cR}{\mathcal{R}}

\newcommand{\cO}{\mathcal{O}}

\newcommand{\cL}{\mathcal{L}}
\newcommand{\cB}{\mathcal{B}}
\newcommand{\cZ}{\mathcal{Z}}
\newcommand{\cY}{\mathcal{Y}}

\newcommand{\oa}{\overline{0}}
\newcommand{\ob}{\overline{1}}

\newcommand{\g}{{\mathfrak{g}}}
\newcommand{\U}{{\rm U}}
\newcommand{\mc}{\mathcal}
\begin{document}
\title{Primitive ideals, twisting functors and star actions\\ for classical Lie superalgebras}

\author{Kevin Coulembier\thanks{corresponding author, E-mail: {\tt Coulembier@cage.ugent.be} Department of Mathematical Analysis -- Ghent University, Krijgslaan 281, 9000 Gent,
Belgium}$\;\,$ and Volodymyr Mazorchuk\thanks{E-mail: {\tt  mazor@math.uu.se}, Department of Mathematics, University of Uppsala, Box 480, SE-75106, Uppsala, Sweden} }
\date{}

\maketitle

\begin{abstract}
We study three related topics in representation theory of classical Lie superalgebras. The first one is 
classification of primitive ideals, i.e. annihilator ideals of simple modules, and inclusions between 
them. The second topic concerns Arkhipov's twisting functors on the BGG category $\cO$. The third 
topic addresses deformed orbits of the Weyl group. These take over the role of the usual Weyl group orbits 
for Lie algebras, in the study of primitive ideals and twisting functors for Lie superalgebras.
\end{abstract}


\section{Introduction}

The problem of classification of all modules over a semisimple finite dimensional Lie algebras is very well-known
to be wild in general. Even the problem of classification of all simple modules seems to be for the
moment out of reach apart from the smallest Lie algebra $\mathfrak{sl}(2)$. This motivates the study of general
rough invariants of simple modules. One of the most natural such invariants is the annihilator of a simple module,
which is a primitive ideal of the universal enveloping algebra. By now the structure of the primitive spectrum of 
the universal enveloping algebra of a semisimple finite dimensional {\em Lie algebra} is well-understood. In \cite{MR0430005} Duflo proved that every primitive ideal is the annihilator ideal of a simple highest weight module. Later on Borho, Dixmier, Garfinkle, Jantzen, Joseph and Vogan completed the classification by describing the sets of highest weights for which these annihilator ideals coincide, see e.g. Chapter 5 and 14 in \cite{MR0721170}. Also all inclusions between primitive ideals have been classified, see e.g. \cite{MR0575938} or Sections 14.15 and 16.4 in \cite{MR0721170}. 

For finite dimensional {\em Lie superalgebras} the situation is much less understood, despite the fact that several
of the above results are generalised to some important cases. For classical Lie superalgebras (see \cite{MR0519631, MR2906817}), an analogue of Duflo's result was proved by Musson in \cite{MR1149625}. For classical Lie superalgebras of type I the classification of primitive ideals was then completed by Letzter in \cite{MR1362685} who showed  that there is a natural bijection between the primitive ideals of a Lie superalgebra and those of its underlying Lie algebra. For inclusions between primitive ideals the situation is even more unclear, so far it has only been studied for the special cases $\mathfrak{sl}(2|1)$ and $\mathfrak{osp}(1|2n)$ in \cite{MR1231215,MR1479886}.

For semisimple {\em Lie algebras}, the action of the centre of the universal enveloping algebra separates simple highest weight modules in different orbits of the Weyl group. Therefore classical results on primitive ideals are
usually formulated in terms of combinatorics of the Weyl group orbits.
For {\em Lie superalgebras}, only typical orbits are similarly separated from each other and from other (atypical) orbits. Atypical central characters correspond, in turn,  to an infinite number of different Weyl group orbits. It is therefore not a priori clear whether the orbits of the Weyl group are the correct structures to e.g. describe inclusions between primitive ideals. In \cite{Dimitar}, Gorelik and Grantcharov use what they call a {\em star action}  to deform orbits of the Weyl group for the Lie superalgebra $\mathfrak{q}(n)$ in order to classify bounded highest weight modules (the latter problem is closely connected to the classification of a certain class of ``large'' primitive ideals).

In the present paper we introduce analogous star actions for basic classical Lie superalgebras using different ideas than the ones used in \cite{Dimitar} for $\mathfrak{q}(n)$, namely, based on Serganova's notion of odd reflections and the naturally defined reflection for an even simple root. We show that the star actions for basic classical Lie superalgebras and for $\mathfrak{q}(n)$ are naturally related to twisting functors and inclusions between annihilator ideals of simple highest weight modules. All those star actions lead in general not to an action of the Weyl group, but to an action of an infinite Coxeter group which projects onto the Weyl group.

This concept of our star actions for basic classical Lie superalgebras leads to several non-equivalent star actions for one superalgebra. This indicates that there will exist more inclusions between primitive ideals for Lie superalgebras than for Lie algebras. An important difference between basic classical Lie superalgebras of type I and type II is that for those of type I the usual Weyl group action is a particular choice of the star action, whereas for type II the latter only holds for typical weights. For both types we prove that, for weights sufficiently far away from the walls of the Weyl chamber which we call (weakly) generic, all star actions coincide and lead to an action of the Weyl group. We similarly prove that in the generic region the star action for $\mathfrak{q}(n)$ leads to an action of the Weyl group.

We obtain a full classification of primitive ideals and their inclusions in the generic region based on the star action. For basic classical Lie superalgebras of type II, this implies that the description of these inclusions is not given by the usual undeformed Weyl group action. For classical Lie superalgebras of type I, our methods lead to more conclusive results and we rederive Letzter's bijection, with the addition that all inclusions between primitive ideals for the Lie algebras are preserved under this bijection. As noted before, different choices of star actions provide a means to obtain {\em additional} inclusions, which we demonstrate explicitly for singly atypical characters, confer \cite{MR1092559, MR1063989}. Therefore star actions are important, on the one hand, to describe the well-behaved structure in the generic region for type II and, on the other hand, to explore the more complicated behaviour close to the walls of the Weyl chambers for both types.

A usual way to study inclusions between annihilator ideals for simple highest weight modules over {\em Lie algebras}
is to use  Joseph's completion functors and Arkhipov's twisting functors, see \cite{MR2032059,MR2115448,MR2331754}
for more details on various versions of these functors. For example, the results of \cite{MR2032059} directly connect the study of annihilator ideals with Kazhdan-Lusztig combinatorics. This will be made explicit in the proof of Proposition \ref{qnint}. Twisting functors have already been introduced in \cite{CMW} for $\mathfrak{gl}(m|n)$ and the special case of simple even roots, in order to obtain equivalences of non-integral blocks for category $\cO$. In the present paper we study twisting functors for classical Lie superalgebras in full generality. In particular, we extend the result in \cite{CMW} on the equivalence of categories between different non-integral blocks. This is closely related to the star action mentioned above, as it turns out that in the
context of such equivalences the role of the usual Weyl group is taken over by the new star action. We generalize
to the Lie superalgebra context most of the classical results on twisting functors, including their action on simple
highest weight modules, Verma modules and projective modules; description of the derived functor and cohomology functors; and applications to the study of the cohomology of simple highest weight modules.

The paper is organised as follows. In Section \ref{secpre} we review necessary notions for classical Lie superalgebras which are relevant to the rest of the paper. In Section \ref{smallex} we extract some results of \cite{MR2742017, MR1231215, MR1479886} which provide a full classification for the inclusions between primitive ideals for the 
small rank Lie superalgebras $\mathfrak{osp}(1|2)$, $\mathfrak{q}(2)$ and $\mathfrak{sl}(2|1)$. These serve as useful illustrations but also as part of the derivation of the results for the other Lie superalgebras in further chapters. In Section \ref{sectech} we state some general technical facts about annihilator ideals, which are well-known for Lie algebras and which extend naturally to Lie superalgebras. In Section \ref{sectwist} we define twisting functors on category $\cO$ and study their properties. A useful fundamental property, which will imply that many classical results carry over, is that they intertwine the restriction and induction functors between category $\cO$ for the Lie superalgebra and its underlying Lie algebra. We study the action of twisting functors on simple and Verma modules, prove that the derived functor leads to an auto-equivalence of the bounded derived category $\cD^b(\cO)$ and establish their relation with annihilator ideals. In Section \ref{sectyp} we use the equivalences of categories for (strongly) typical blocks, proved by Gorelik in \cite{MR1862800} for basic classical Lie superalgebras, by Serganova in \cite{MR1943937} for $\mathfrak{p}(n)$ and by Frisk and the second author in \cite{Frisk} for $\mathfrak{q}(n)$, to classify inclusions between primitive ideals for typical blocks. In Section \ref{secgen} we introduce the notion of generic modules similar to \cite{MR1309652, MR1201236} and obtain some preliminary results for such modules. Section \ref{secstar} is devoted to the star actions. In Subsection \ref{starbasic} we give the definition of a star action for basic classical Lie superalgebras and illustrate their connection with twisting functors and primitive ideals. Then we prove that all star actions become identical and regular in the generic region. In Subsection~\ref{exstarosp} we focus on a central example of a star action for $\mathfrak{osp}(m|2n)$. In Subsection \ref{starqn} we prove that we can extend our results on star actions of basic classical Lie superalgebras to the one for $\mathfrak{q}(n)$ as defined in \cite{Dimitar}. In Sections \ref{secqn} and \ref{primosp} we study primitive ideals of respectively $\mathfrak{q}(n)$ and $\mathfrak{osp}(m|2n)$. The main result is a full classification of primitive ideals, and all inclusions between them, in the generic region. In Section \ref{secI} we focus on primitive ideals of classical Lie superalgebras of type I. We reobtain Letzter's bijection and describe inclusions between primitive ideals. Finally, we make some remarks on the annihilator ideals of Verma modules in Section~\ref{secVerma}. The main conclusion is that inclusions between such annihilators are naturally described in terms of the usual Weyl group orbits, not by the star action.

\section{Preliminaries}
\label{secpre}

We work over $\mC$ and set $\mN = \{1,2,3,\dots\}$ and $\mZ_+ = \{0,1,2,3\dots\}$. For a Lie (super)algebra $\fa$, we denote by $ U(\fa)$ the corresponding universal enveloping (super)algebra and by $\cZ(\fa)$ the centre of $U(\fa)$. Let $\fg = \fg_{\oa}+\fg_{\ob}$ be a Lie superalgebra. From now on we assume that $\fg$ is {\em classical} in the sense that $\fg_{\oa}$ is a finite dimensional reductive Lie algebra and $\fg_{\ob}$ is a semi-simple finite dimensional $\fg_{\oa}$-module. We do not require $\fg$ to be simple. The concrete subset of the classical Lie superalgebras we will consider is given by the following:
\begin{eqnarray}\label{list}&&\mathfrak{gl}(m|n),\,\,\mathfrak{sl}(m|n),\,\,\mathfrak{psl}(n|n),\,\,\mathfrak{osp}(m|2n),\,\, D(2,1;\alpha),\,\, G(3),\,\, F(4),\\
\nonumber
&&\mathfrak{p}(n),\,\,\widetilde{\mathfrak{p}}(n),\,\,\mathfrak{q}(n),\,\,\mathfrak{sq}(n),\,\, \mathfrak{pq}(n)\mbox{ and }\mathfrak{psq}(n).\end{eqnarray}
For the definition of these algebras except the last four, see \cite{MR2906817}. The definition of the last four Lie superalgebras (called superalgebras of {\em queer} or {\em $Q$-type}) can be found in Section 3.2 of \cite{MR2282179}. When the statements or proofs for the algebras of $Q$-type differ from the other cases we will write them only explicitly for $\mathfrak{q}(n)$. However, they also hold for the other three, where it should be taken into account that the notion of $\underline{\mbox{strong}}$ typicality for $\mathfrak{sq}(n)$ and $\mathfrak{psq}(n)$ differs from that for $\mathfrak{q}(n)$. The Lie superalgebras $\mathfrak{p}(n)$ and $\widetilde{\mathfrak{p}}(n)$ are the ones of {\em strange} type.

The classical Lie superalgebras of {\em type I} are the ones that possess a $\mZ$-gradation of the form
\[\fg=\fg_{-1}\oplus\fg_0\oplus\fg_1\qquad\mbox{with}\qquad \fg_{\oa}=\fg_0\mbox{ and }\fg_{\ob}=\fg_{-1}\oplus\fg_1.\]
This restricts to $\mathfrak{gl}(m|n),\,\,\mathfrak{sl}(m|n),\,\,\mathfrak{psl}(n|n),\,\,\mathfrak{osp}(2|2n),\,\, \mathfrak{p}(n)$ and $\widetilde{\mathfrak{p}}(n)$. The others are called the classical Lie superalgebras of {\em type II}. Classical Lie superalgebras are called {\em basic} if they have an even invariant form, see e.g. \cite{MR0519631, MR2906817}, these are the ones on the first line of the list \eqref{list}.

We fix a Cartan subalgebra $\fh$ and Borel subalgebra $\fb$, see Chapter 3 in \cite{MR2906817}. Note that all algebras we consider, except those of $Q$-type, have a purely even Cartan subalgebra $\fh=\fh_{\oa}$. The set of roots corresponding to $\fh_{\oa}$ is denoted by $\Delta$, so $\fg=\fh+\sum_{\alpha\in\Delta} \fg_{\alpha}$. The set of positive (negative) roots corresponding to $\fb$ is denoted by $\Delta^+$ ($\Delta^-$), the subset of simple positive roots by $\Pi$. The sets of even and odd roots are denoted respectively by $\Delta_{\oa}$ and $\Delta_{\ob}$ with similar notation for positive and negative roots. We have the corresponding triangular decomposition
\[\fg=\fn^-\oplus\fh\oplus\fn^+\]
with $\fn^+=\oplus_{\alpha\in\Delta^+}\fg_\alpha$, $\fn^-=\oplus_{\alpha\in\Delta^-}\fg_\alpha$ and $\fb=\fh\oplus\fn^+$. 

The Weyl group $W$ acting on $\fh_{\oa}^\ast$ is the Weyl group $W(\fg_{\oa}:\fh_{\oa})$ of the underlying Lie algebra. This group is generated by the simple reflections $s_{\alpha}$ for $\alpha$ simple in $\Delta_{\oa}^+$. The basis of simple roots in $\Delta_{\oa}^+$ is denoted by $\Pi_{\oa}$. This basis is not to be confused with the the subset of simple roots in $\Delta^+$ which are even, since, in general, $\Pi_{\oa}\not=\Pi\cap\Delta_{\oa}$. The longest element of $W$ is denoted by $w_0$.

We define 
\[\rho_{\oa}=\frac{1}{2}\sum_{\alpha\in\Delta_{\oa}^+}\alpha,\qquad \rho_{\ob}=\frac{1}{2}\sum_{\alpha\in\left(\Delta_{\ob}^+\cap (-\Delta_{\ob}^-)\right)}\alpha\qquad\mbox{ and }\quad\rho=\rho_{\oa}-\rho_{\ob}.\]

We follow the convention of e.g. \cite{MR1479886, MR2906817} to denote the $\rho$-shifted action of the Weyl group by $w\cdot\lambda=w(\lambda+\rho)-\rho$ and the $\rho_{\oa}$-shifted action by $w\circ\lambda=w(\lambda+\rho_{\oa})-\rho_{\oa}$. Note that in the specific case $\fg=\mathfrak{q}(n)$ this implies that $w\cdot\lambda=w\lambda$, which leads to a different notational convention from e.g. \cite{Frisk, Dimitar}.

The {\em Verma module} with highest weight $\lambda\in\fh^\ast_{\oa}$ is denoted by
\[M^{(\fb)}(\lambda)=U(\fg)\otimes_{U(\fb)}L_{\fb}(\lambda),\]
with $L_{\fb}(\lambda)$ a $\fb$-module which is a simple $\fh$-module with trivial $\fn^+$-action and on which $\fh_{\oa}^\ast$ acts through the weight $\lambda$. This module $L_{\fb}(\lambda)$ is uniquely defined up to parity change and is one dimensional if $\fg$ is not of $Q$-type. Only for algebras of $Q$-type and only for
certain choices of $\lambda$ the module $L_{\fb}(\lambda)$ is invariant (up to isomorphism) under the parity change
(see e.g. \cite{MR2282179} for details). Except for some algebras of $Q$-type it is possible to distinguish between  a Verma module and its parity changed by their super-character.  Since it is not relevant for the purposes of this paper, we will not make any explicit distinction between the two Verma modules with the same highest weight. The unique simple quotient of a Verma module is the {\em simple module} with highest weight $\lambda$. It is denoted by $L^{(\fb)}(\lambda)$, which is subject to the same parity change issues as the corresponding Verma module. In most cases we will leave out the explicit reference to the Borel subalgebra and denote the Verma module and its simple quotient by $M(\lambda)$ and $L(\lambda)$, respectively. We denote by $M_{\overline{0}}(\lambda)$ and $L_{\overline{0}}(\lambda)$,
respectively, the Verma module with highest weight $\lambda$ and the corresponding simple highest weight quotient
for the Lie algebra $\mathfrak{g}_{\overline{0}}$.

For any $\alpha\in\Delta_{\oa}^+$ we define $\alpha^\vee=2\alpha/ \langle \alpha,\alpha\rangle$. A weight $\lambda\in\fh^\ast_{\oa}$ is called {\em integral} if there is a finite dimensional $\fg$-module for which the corresponding weight space is nonzero. The set of integral weights is denoted by $\cP$. The subset of integral {\em dominant} weights is defined as
$$\cP^+=\{\lambda\in\fh_{\oa}^\ast\,\,|\, \dim L(\lambda) <\infty\}.$$
Note that for an integral weight $\lambda$, the condition $\langle\lambda,\alpha^\vee\rangle\in\mZ$ for $\alpha\in\Delta_{\oa}^+$ is always satisfied. Furthermore, we point out that, contrary to $\cP$, $\cP^+$ depends on $\fb$.

For non-integral weights $\lambda$ we define the subsystem of roots (known as the integral root system)
\begin{equation}\label{deltalambda}\Delta_{\oa}(\lambda)=\{\alpha\in\Delta_{\oa}\,|\,\langle \lambda,\alpha^{\vee}\rangle\in\mZ\},\end{equation}
with corresponding integral Weyl group $W_\lambda\subset W$ generated by $s_\alpha$ for $\alpha\in\Delta_{\oa}(\lambda)$. This integral Weyl group can equivalently be defined as
\begin{equation}\label{intWeyl}W_\lambda=\{w\in W\,|\, w(\lambda)-\lambda\in \cP\},\end{equation}
see Section 3.4 in \cite{MR2428237}. The basis of simple roots in $\Delta(\lambda)\cap \Delta^+$ is denoted by $\Pi_\lambda$. Then the set 
\begin{equation}
\label{coset}W^\lambda= \{w\in W\,|\,w(\Pi_\lambda)\subset\Delta^+\}\end{equation}
is a set of left coset representatives for $W_\lambda$ in $W$, see e.g. Lemma 15.3.6 in \cite{MR2906817}.

For each classical Lie superalgebra there is a distinguished Borel subalgebra (distinguished system of positive roots) as defined by Kac in \cite{MR0519631}. In the distinguished system of positive roots for Lie superalgebras of type I, each simple root in $\Delta_{\oa}^+$ is also simple in $\Delta^+$.

For $\alpha$ a root simple in $\Delta_{\oa}^+$ we say that a module $M$ is $\alpha$-free (respectively $\alpha$-finite) if for a non-zero $Y\in(\fg_{\oa})_{-\alpha}$ the action of $Y$ is injective (respectively locally finite) on $M$. Note that for all algebras in the list \eqref{list} we have $\dim(\fg_{\oa})_{-\alpha}=1$. A simple module is either $\alpha$-finite or $\alpha$-free. The following property then follows immediately from reduction to $\mathfrak{sl}(2)$.

\begin{lemma}\label{lemfinfree}
Consider $\lambda\in\fh^\ast_{\oa}$ and $\alpha$ a simple root in $\Delta^+_{\oa}$. For all classical Lie superalgebras $\fg$ we have that
\[L_{\oa}(\lambda)\mbox{ is $\alpha$-free }\qquad \Rightarrow \qquad L(\lambda)\mbox{ is $\alpha$-free}. \]
This is equivalent to
\[L(\lambda)\mbox{ is $\alpha$-finite}\qquad \Rightarrow \qquad L_{\oa}(\lambda)\mbox{ is $\alpha$-finite}. \]
If $\alpha$ or $\alpha/2$ is also simple in $\Delta^+$ and $\left(\fg_{\ob}\right)_{\alpha}=0$, the implications can be reversed. In particular, this is always the case if $\fg$ is of type I with distinguished system of positive roots.
\end{lemma}

As in \cite{MR2742017, MR1149625, MR1231215, MR1479886}, for any $\lambda\in\fh_{\oa}^\ast$ we use the notation
\[I(\lambda)=\Ann_{U(\fg_{\oa})}L_{\oa}(\lambda)\quad\mbox{and}\quad J(\lambda)=\Ann_{U(\fg)}L(\lambda)\]
for primitive ideals of the underlying Lie algebra and the Lie superalgebra, respectively. Note that the annihilator ideal of a $\fg$-module  is the same as the annihilator ideal of the parity reversed module. The possible ambiguity in the definition of $L(\lambda)$ is therefore not reflected in the definition of $J(\lambda)$.

We mention the following result on primitive ideals for $\fg_{\oa}$, which can be obtained immediately from Corollary 2.13 in \cite{MR0453826} or Lemmata 5.4 and 5.6 in \cite{MR0721170}.

\begin{lemma}\label{translclass}
Consider two $\fg_{\oa}$-dominant weights $\Lambda_1,\Lambda_2$ with $\Lambda_1-\Lambda_2\in\cP$ and $u,w\in W$. Then we have
\[I(u\circ\Lambda_1)\subseteq I(w\circ\Lambda_1)\quad\Leftrightarrow\quad I(u\circ\Lambda_2)\subseteq I(w\circ\Lambda_2).\]
\end{lemma}

We briefly review the notion of (strong) typicality for basic classical Lie superalgebras and for $\mathfrak{q}(n)$. Since we do not need the explicit definitions for atypical roots for Lie superalgebras of strange type, we just refer to \cite{MR1201236, MR1943937}. 

Let $\fg$ be basic classical. {\em Typical} weights $\lambda\in\fh^\ast$ are weights which satisfy $\langle \lambda+\rho,\gamma\rangle\not=0$ for any isotropic root $\gamma\in\Delta^+_{\ob}$. {\em Strongly typical} weights are those which satisfy $\langle \lambda+\rho,\gamma\rangle\not=0$ for any odd root $\gamma\in\Delta^+_{\ob}$, see \cite{MR1862800}. This implies that for all basic classical Lie superalgebras, except $\mathfrak{osp}(2d+1|2n)$ and $G(3)$, strongly typical and typical are the same concept. An isotropic root $\gamma$ for which $\langle \lambda+\rho,\gamma\rangle\not=0$ is called an {\em atypical} root for $\lambda$. It follows from the Harish-Chandra isomorphism for basic classical Lie superalgebras (see e.g. Section 13.1 in \cite{MR2906817}) that the only weights $\mu\in\fh^\ast$ for which $L(\mu)$ admits the same central character as $L(\lambda)$ for $\lambda$ typical are in the $\rho$-shifted Weyl group orbit of $\lambda$. Therefore the notion of a (strongly) typical central character arises. In \cite{MR1862800} Gorelik proved that strongly typical blocks in the category of $\fg$-modules are equivalent to blocks of the category of $\fg_{\oa}$-modules. For a strongly typical central character $\chi:\cZ(\fg)\to\mC$ there exists a {\em perfect mate}, which is a central character $\widetilde{\chi}:\cZ(\fg_{\oa})\to\mC$ such that the functors 
\begin{equation}\label{eqGorelik}\left(\Ind -\right)_{\chi}\quad\mbox{and}\quad \left(\Res -\right)_{\widetilde{\chi}}\end{equation}
are the functors inducing this equivalence of categories.

Now we consider $\fg=\mathfrak{q}(n)$, then $\Delta_{\oa}^+=\Delta_{\ob}^+=\{\epsilon_i-\epsilon_j|1\le i <j\le n\}$, so $\rho=0$. A weight $\lambda\in\fh_{\oa}^\ast$ is {\em atypical} with respect to $\alpha\in\Delta_{\ob}^+$ if $\langle\lambda,\overline{\alpha}\rangle=0$, with $\overline{\epsilon_i-\epsilon_j}=\epsilon_i+\epsilon_j$. The weight $\lambda$ is {\em typical} if it has no atypical roots. For the explicit definition of strongly typical weights see e.g. \cite{Frisk}. In that paper the result from \cite{MR1862800} on equivalences of categories for strongly typical blocks of basic classical Lie superalgebras was extended to strongly typical regular blocks of $\mathfrak{q}(n)$.

An important role will be played by odd reflections, see e.g. Section 3 in \cite{MR2743764} or Section 3.5 in \cite{MR2906817}. Therefore we briefly review the concept below, we restrict to $\underline{\mbox{basic}}$ classical Lie superalgebras here. Two systems of positive roots (of which we denote one by $\Delta^+$) which have the same even positive roots ($\Delta_{\oa}^+$) are called {\em adjacent} if there is an isotropic root $\gamma$ simple in $\Delta^+$ such that the other system of positive roots is equal to $\left(\Delta^+\backslash \gamma\right)\cup \{- \gamma\}$. Then we also say that the second system of positive roots is obtained from $\Delta^+$ by application of the {\em odd reflection} corresponding to $\gamma$.

Now consider two arbitrary systems of positive roots $\Delta^+$ and $\hat{\Delta}^+$, such that $\Delta_{\oa}^+=\hat{\Delta}_{\oa}^+$. Correspondingly we have two Borel subalgebras $\fb$ and $\hat{\fb}$, with $\fb_{\oa}=\hat{\fb}_{\oa}$. Theorem 3.1.3 in \cite{MR2906817} then states that there exists an ordered set of odd isotropic roots $\{\gamma^{(1)},\cdots,\gamma^{(p)}\}$ such that the corresponding sequence of odd reflections is well defined and, starting from $\Delta^+$, eventually yields $\hat{\Delta}^+$. In particular, the root $\gamma^{(1)}$ is positive simple in $\Delta^+$, the root $\gamma^{(2)}$ is positive simple in the second system of positive roots defined as $\left(\Delta^+\backslash \gamma^{(1)}\right)\cup \{- \gamma^{(1)}\}$ and so on. The procedure is repeated until the odd reflection with respect to $\gamma^{(p)}$ yields $\hat{\Delta}^+$.

How highest weights of the same highest weight module in different systems of positive roots (with the same $\Delta_{\oa}^+$) are related is summarised in the following lemma, confer Lemma 0.3 in \cite{MR1201236}.

\begin{lemma}
\label{oddrefl}
Let $\fg$ be a basic classical Lie superalgebra with two Borel subalgebras $\fb$ and $\hat{\fb}$, such that $\fb_{\oa}=\hat{\fb}_{\oa}$, which are related by the ordered set of odd isotropic roots $\{\gamma^{(1)},\cdots,\gamma^{(p)}\}$ as above. The highest weight modules of a simple module in the different systems of positive roots,
\[L^{(\fb)}(\lambda)\cong L^{(\hat{\fb})}(\hat{\lambda}),\]
satisfies
\[\hat{\lambda}=\lambda+\sum_{i=1}^q\gamma_i+\rho-\hat{\rho}=\lambda-\sum_{j=1}^p\gamma^{(j)} +\sum_{i=1}^q\gamma_i, \]
where the ordered subset $\{\gamma_1,\cdots,\gamma_q\}\subset \{\gamma^{(1)},\cdots,\gamma^{(p)}\}$ is given by the following recursive algorithm. The root $\gamma_1$ is equal to $\gamma^{(i_1)}$ for $i_1$ the smallest $i_1$ such that $\langle\lambda+\rho,\gamma^{(i)}\rangle=0$. The root $\gamma_s$ is equal to $\gamma^{(i_s)}$ for $i_s$ the smallest $i >i_{s-1}$ such that $\langle\lambda+\gamma_1+\cdots+\gamma_{s-1}+\rho,\gamma^{(i)}\rangle=0$.
\end{lemma}

\begin{proof}
Consider the Verma module $M^{(\fb)}(\lambda)$. The dimension of the weight space of weight $\lambda-\gamma^{(1)}$ is one, since $\gamma^{(1)}$ is simple in $\Delta^+$. We denote a non-zero vector is this weight space by $x$. This $x$ is a highest weight vector if and only $\langle \lambda,\gamma^{(1)}\rangle=\langle \lambda+\rho,\gamma^{(1)}\rangle=0$. 

Now we consider $M^{(\fb)}(\lambda)$ in the system of positive roots given by $(\Delta^+\backslash \gamma^{(1)})\cup {-\gamma^{(1)}}$. In this system, the weight of $x$ is the highest weight of $M^{(\fb)}(\lambda)$. 

If $x$ is not a highest weight vector
(with respect to $\fb$), it generates $M^{(\fb)}(\lambda)$, showing that $M^{(\fb)}(\lambda)$ is still a Verma module with respect to the new Borel subalgebra, but now with highest weight $\lambda-\gamma^{(1)}$. Since the simple module is isomorphic to the quotient of this Verma module with respect to its maximal submodule, the new highest weight is $\lambda-\gamma^{(1)}$. If $x$ is a highest weight vector (with respect to $\fb$), it is factored out in $L^{(\fb)}(\lambda)$, so $L^{(\fb)}(\lambda)$ is a quotient of the module $N:=M^{(\fb)}(\lambda)/(U(\fg)x)$. The highest weight of $N$ is still $\lambda$ in the new system of positive roots.

From these considerations one obtains that the highest weight of the module $L^{(\fb)}(\lambda)$ in the second system of positive roots is as follows:
\begin{itemize}
\item $\lambda-\gamma^{(1)}$ if $\langle \lambda+\rho,\gamma_1\rangle\not=0$,
\item $\lambda$ if $\langle \lambda+\rho,\gamma_1\rangle=0$;
\end{itemize}
which corresponds to the proposed expression if $p=1$. Then we proceed by induction. We assume the formula is valid for the system of roots $\widetilde{\Delta}^+$, which is the one obtained from $\Delta^+$ by applying the odd reflections corresponding to $\{\gamma^{(1)},\cdots,\gamma^{(t)}\}$ with $1\le t  <p$. If $L^{(\fb)}(\lambda)=L^{(\widetilde{\fb})}(\widetilde{\lambda})$ with $\widetilde{\lambda}=\lambda+\sum_{i=1}^{q'}\gamma_i+\rho-\widetilde{\rho}$, then the highest weight in the new system of positive roots depends on
\[\langle \widetilde{\lambda},\gamma^{(t+1)}\rangle=\langle \widetilde{\lambda}+\widetilde{\rho},\gamma^{(t+1)}\rangle=\langle\lambda+\sum_{i=1}^{q'}\gamma_i+\rho,\gamma^{(t+1)}\rangle, \]
from which the claim follows.
\end{proof}

\begin{corollary}
\label{oddtypical}
Let $\fg$ be a basic classical Lie superalgebra with two Borel subalgebras $\fb$ and $\hat{\fb}$ such that $\fb_{\oa}=\hat{\fb}_{\oa}$. If $\lambda\in\fh^\ast$ is typical, then we have the following isomorphism of Verma modules for the different Borel subalgebras
\[M^{(\fb)}(\lambda)=U(\fg)\otimes_{U(\fb)}\mC_\lambda \cong U(\fg)\otimes_{U(\hat{\fb})}\mC_{\lambda+\rho-\hat{\rho}}= M^{(\hat\fb)}(\lambda+\rho-\hat{\rho})\]
and the corresponding isomorphism of simple modules
\[L^{(\fb)}(\lambda)\cong L^{(\hat{\fb})}(\lambda+\rho-\hat{\rho}).\]
\end{corollary}

We will use the BGG category $\cO$ for Lie superalgebras, see e.g. \cite{preprint, MR2906817}. One of the equivalent ways of defining this category is as the subcategory in $\fg$-smod consisting of all objects which are mapped to $\cO_{\oa}$ under $\Res$, with $\cO_{\oa}$ being the corresponding BGG category for $\fg_{\oa}$, see \cite{MR0407097, MR2428237}. The bounded derived category of $\cO$ is denoted by $\cD^b(\cO)$. The indecomposable projective cover of $L(\lambda)$ in $\cO$ is denoted by $P(\lambda)$.

In the following paragraphs we study decompositions of $\cO$ into subcategories related to central characters, we leave out the strange algebras $\mathfrak{p}(n)$ and $\widetilde{\mathfrak{p}}(n)$. The subcategory of $\cO$ corresponding to a central character $\chi:\cZ(\fg)\to\mC$ is denoted by $\cO_\chi$. For an atypical or a typical non-integral central character, $\cO_{\chi}$ is not indecomposable. We introduce $$\fh^\ast_{\oa,{\rm dom}}=\{\lambda\in\fh_{\oa}^\ast\,|\,\lambda \mbox{ is maximal with respect to the $\rho$-shifted action of }W_\lambda\}.$$
The category $\cO_\lambda$ is then defined as the Serre subcategory of $\cO$ generated by $L(\mu)\in \cO_{\chi_\lambda}$ with $\mu\in\lambda+\cP$. According to the description of the integral Weyl group in equation \eqref{intWeyl} and the condition on two weights to have the same central character, which can be found in Section 13.1 in \cite{MR2906817} for basic classical Lie superalgebras and in \cite{MR0706205} for the algebras of $Q$-type, we have
\begin{displaymath}
\begin{array}{rcl}
L(\mu)\in \cO_\lambda&\Leftrightarrow&
\mu=w\cdot(\lambda-\sum_j k_j\gamma_j) \quad\mbox{ for some }w\in W_\lambda \,\mbox{ and a maximal} \\
&& \mbox{set } \{\gamma_j\} \mbox{ of mutually orthogonal atypical roots of }\lambda.
\end{array}
\end{displaymath}
This shows that our definition of $\cO_\lambda$ coincides with the particular case of $\mathfrak{gl}(m|n)$ in \cite{CMW}. From the definition and the property $\mZ\Delta\subset\cP$ it is also clear that there are no extensions between modules in strictly different categories $\cO_\lambda$ and $\cO_{\lambda'}$.

We define an equivalence relation as follows:
$$\mbox{for } \lambda,\mu\in\fh_{\oa,{\rm dom}}^\ast\quad\mbox{we set}\quad  \lambda\sim \mu\, \mbox{ if }\,\, \lambda-\mu\in\cP \,\mbox{ and } \,\,\chi_\lambda=\chi_\mu. $$ 
In other words $\lambda\sim\mu$ iff $\cO_\mu= \cO_\lambda$. This yields the decomposition $\cO=\bigoplus_{\lambda\in\fh^\ast_{\oa,{\rm dom}}/\sim}\cO_\lambda$, where an arbitrary representative in $\fh^\ast_{\oa,{\rm dom}}$ is chosen for each element of the quotient $\fh^\ast_{\oa,{\rm dom}}/\sim$. In particular, we also have
\begin{equation}\label{decompnonint}\cO_{\chi_\lambda}=\bigoplus_{w\in W^\lambda}\cO_{w\cdot\lambda}.\end{equation}
If $\fg$ is a reductive Lie algebra, the equivalence relation $\sim$ becomes trivial and the usual decomposition into indecomposable blocks is obtained.

For the superalgebras $\mathfrak{p}(n)$ and $\widetilde{\mathfrak{p}}(n)$, the corresponding universal enveloping algebra does not have a relevant centre $\cZ(\fg)$. Whenever $\cO_\chi$ is used we silently assume that it is just given by $\cO$ for those algebras. The subcategory $\cO_\lambda$ is thus the Serre subcategory generated by the simple modules $\{L(\mu)\,|\,\mu\in \lambda+\cP\}.$


\section{Primitive ideals for $\mathfrak{osp}(1|2)$, $\mathfrak{q}(2)$ and $\mathfrak{sl}(2|1)$}
\label{smallex}

In this section we review the classification of primitive ideals and inclusions between them for $\mathfrak{osp}(1|2)$, $\mathfrak{q}(2)$ and $\mathfrak{sl}(2|1)$, as can be found in \cite{MR2742017, MR1231215, MR1479886}. The underlying Lie algebras for these Lie superalgebras are $\mathfrak{sl}(2)$ or $\mathfrak{gl}(2)$, so the Weyl groups are all isomorphic to $\mZ_2$.
\subsection{Primitive ideals for $\mathfrak{osp}(1|2)$}

The Cartan subalgebra of $\mathfrak{osp}(1|2)$ is one-dimensional. There is one even positive root $\alpha$ and one odd positive root $\alpha/2$. All weights of $\mathfrak{osp}(1|2)$ are typical, so all inclusions between primitive ideals can be obtained from the results in the subsequent Section \ref{sectyp}. These inclusions also follow from the results of Musson in Theorem B and Theorem 1.4 in \cite{MR1479886} or Pinczon in \cite{MR1060845}.

\begin{proposition}
\label{osp12}
Consider $\fg=\mathfrak{osp}(1|2)$ and $s$ the only non-trivial element of the Weyl group. The only inclusions between primitive ideals are given by
\begin{itemize}
\item $\Ann_{ U(\fg)}L(\mu)=\Ann_{ U(\fg)}L(\lambda)$ for $\lambda$ not integral and $\mu=s\cdot\lambda$,
\item $\Ann_{ U(\fg)}L(\mu)\subsetneq\Ann_{ U(\fg)}L(\lambda)$ for $\lambda$ integral dominant and $\mu=s\cdot\lambda$.
\end{itemize}
\end{proposition}

\subsection{Primitive ideals for $\mathfrak{q}(2)$}
\label{subsecq2}

The Cartan subalgebra $\fh$ of $\mathfrak{q}(2)$ has super dimension $2|2$. There is one even positive root $\epsilon_1-\epsilon_2$ which is also an odd root. A weight $\lambda\in\fh_{\oa}^\ast$ is integral if $\langle \lambda,\epsilon_1-\epsilon_2\rangle\in\mZ$, it is dominant if $\langle \lambda,\epsilon_1-\epsilon_2\rangle \in\mN$ or if $\lambda=0$.

Proposition 8 in \cite{MR2742017} classifies primitive ideals for $\mathfrak{q}(2)$. Even though it is not stated explicitly in the proposition, inclusions between primitive ideals can also be derived from the proof. This can be summarised in terms of the star action defined as follows: $$s\ast\lambda=s\lambda\mbox{ if }\langle \lambda,\epsilon_1+\epsilon_2\rangle\not=0\quad\mbox{ and }\quad s\ast\lambda=s\lambda-\epsilon_1+\epsilon_2\mbox{ if }\langle \lambda,\epsilon_1+\epsilon_2\rangle=0,$$ see Section \ref{starqn} for more details on this star action.

\begin{proposition}
\label{Primq2}
Consider $\fg=\mathfrak{q}(2)$ and $s$ the only non-trivial element of the Weyl group. The only inclusions between primitive ideals are given by
\begin{itemize}
\item $\Ann_{ U(\fg)}L(\mu)=\Ann_{ U(\fg)}L(\lambda)$ for $\lambda$ not integral and $\mu=s\ast\lambda$,
\item $\Ann_{ U(\fg)}L(\mu)\subsetneq\Ann_{ U(\fg)}L(\lambda)$ for $\lambda$ integral dominant and $\mu=s\ast\lambda$.
\end{itemize}
\end{proposition}

Simple modules over $\mathfrak{pq}(2)$, $\mathfrak{sq}(2)$ and $\mathfrak{psq}(2)$ are discussed in Subsections 3.10, 3.11 and 3.12 of \cite{MR2742017}.

\subsection{Primitive ideals for $\mathfrak{sl}(2|1)$}
The Cartan subalgebra of $\mathfrak{sl}(2|1)$ is $2$-dimensional. The distinguished system of positive roots has simple roots $\epsilon_1-\epsilon_2$ and $\epsilon_2-\delta$. Weights are given by $k_1\epsilon_1+k_2\epsilon_2+l\delta$ with $k_1+k_2+l=0$. Inclusions between primitive ideals for $\mathfrak{sl}(2|1)$ were classified by Musson in Theorem 3.1 in \cite{MR1231215}.

\begin{proposition}
\label{sl21}
Consider $\fg=\mathfrak{sl}(2|1)$ and $s$ the only non-trivial element of the Weyl group. The only inclusions between primitive ideals are given by
\begin{itemize}
\item $\Ann_{ U(\fg)}L(\mu)=\Ann_{ U(\fg)}L(\lambda)$ for $\langle\lambda,\epsilon_1-\epsilon_2\rangle\not\in\mZ$ and $\mu=s\cdot\lambda$,
\item $\Ann_{ U(\fg)}L(\mu)\subsetneq\Ann_{ U(\fg)}L(\lambda)$ for $\langle\lambda+\rho,\epsilon_1-\epsilon_2\rangle\in \mN$ and $\mu=s\cdot\lambda$ and
\item $\Ann_{U(\fg)}L(\epsilon_2-\delta)\subsetneq\Ann_{U(\fg)}L(0)$.
\end{itemize}
\end{proposition}


\section{Technical tools for primitive ideals}
\label{sectech}

For a central character $\chi:\cZ(\fg)\to\mC$ and a module $M\in\cO$ we define
\[M_\chi=\{v\in M\,|\, (z-\chi(z))^k v=0\, \mbox{ for all } z\in \cZ(\fg)\mbox{ for some }k\mbox{ depending on }z \}.\]
Any module in $\cO$ decomposes into such modules corresponding to different central characters. The following result is an immediate generalisation of the ideas in Section 2.7 in \cite{MR0453826} or Section 5.3 in \cite{MR0721170}.

\begin{lemma}
\label{IplusAnn}
Consider a $\fg$-module $M\in\cO$ and a central character $\chi:\mathcal{Z}(\fg)\to\mC$, with the corresponding maximal ideal $\mm_{\chi}=\ker\chi\in \Spec \mathcal{Z}(\fg)$ and the two-sided ideal $I_\chi= U(\fg)\mm_{\chi}$ in $ U(\fg)$. If the submodule $M_\chi$ admits the central character $\chi$, i.e. $\mathfrak{m}_\chi M_\chi=0$, then we have the equality
\[\Ann_{ U(\fg)}M_\chi=I_\chi+\Ann_{ U(\fg)}M.\]
If the condition $\mathfrak{m}_\chi M_\chi=0$ is dropped, we have the more general equality
\[\sqrt{\Ann_{ U(\fg)}M_\chi}=\sqrt{I_\chi+\Ann_{ U(\fg)}M},\]
where $\sqrt{X}$ stands for the intersection of all prime ideals containing $X$.
\end{lemma}

\begin{proof}
The module $M$ decomposes as $M=M_\chi\bigoplus \oplus_{i=1}^k M_{\chi^i}$ for $k$ other central characters $\chi^i$. If $\mm_\chi M_\chi=0$, then $I_\chi\subset \Ann_{U(\fg)}M_\chi$. Therefore
\begin{eqnarray*}I_\chi+\Ann_{U(\fg)}M&=&I_\chi+\left(\cap_{i=1}^k \Ann_{U(\fg)}M_{\chi^i}\,\cap\, \Ann_{U(\fg)}M_\chi\right)\\
&=&\left(I_\chi+\cap_{i=1}^k \Ann_{U(\fg)}M_{\chi^i}\right)\,\cap\, \Ann_{U(\fg)}M_\chi.\end{eqnarray*}
The ideal $I_\chi+\cap_{i=1}^k \Ann_{U(\fg)}M_{\chi^i}$ contains the entire centre $\mathcal{Z}(\fg)$ since $\mm_\chi$ is a maximal ideal. Since $1\in\mathcal{Z}(\fg)$, we have $\left(I_\chi+\cap_{i=1}^k \Ann_{U(\fg)}M_{\chi^i}\right)=U(\fg)$ and the first statement follows.

Now we consider the general case and we set $\chi^0:=\chi$. Let $Q$ be a prime ideal which contains $I_\chi+\Ann_{ U(\fg)}M$. Since $\Ann_{ U(\fg)}M\subset Q$, there is an $0\le i\le k$ for which $\Ann_{U(\fg)}M_{\chi^i}\subset Q$. Furthermore, there is an $l\in\mN$ such that $I_{\chi^i}^l\subset \Ann_{U(\fg)}M_{\chi^i}$, thus $I_{\chi^i}\subset Q$ follows. Since $I_{\chi^0}\subset Q$ and $I_{\chi^0}+I_{\chi^i}=U(\fg)$ if $i\not=0$, we obtain $\Ann_{U(\fg)}M_{\chi}\subset Q$. 

If $Q$ is a prime ideal which contains $\Ann_{ U(\fg)}M_\chi$, then similarly $I_\chi\subset Q$ and naturally $\Ann_{U(\fg)}M\subset \Ann_{U(\fg)}M_\chi\subset Q$, so $Q$ contains $I_\chi+\Ann_{ U(\fg)}M$, which completes the proof.
\end{proof}

This leads immediately to the following conclusion.

\begin{corollary}
\label{corIAnn}
Assume that for two $\fg$-modules we have $\Ann_{U(\fg)}M\subseteq \Ann_{U(\fg)}N$. If for a central character $\chi: \mathcal{Z}(\fg)\to \mC$ we have $M_\chi=0$, then $N_\chi=0$ follows.
\end{corollary}

\begin{proof}
Lemma \ref{IplusAnn} implies that $U(\fg)\subseteq\sqrt{\Ann_{U(\fg)}N_\chi}$. Since $\Ann_{U(\fg)}N_\chi\not= U(\fg)$ would imply that $\Ann_{U(\fg)}N_\chi$ is contained in a non-trivial maximal ideal, the result $N_\chi=0$ follows.
\end{proof}

As in Lemma 15.3.17 in \cite{MR2906817} or in Section 2.4 in \cite{MR0453826}, we have the following result.
\begin{lemma}
\label{Anntens}
For two $\fg$-modules $M$ and $E$ with $E$ finite dimensional and simple, the annihilator ideal of the tensor product $M\otimes E$ is given by
\[\mathbf{\Delta}^{-1}(\Ann_{ U(\fg)}M\otimes  U(\fg)+ U(\fg)\otimes \Ann_{ U(\fg)}E)\]
where $\mathbf{\Delta}: U(\fg)\to  U(\fg)\otimes U(\fg)$ is the comultiplication of the Hopf superalgebra $ U(\fg)$.
\end{lemma}

The following lemma corresponds to Lemma 7.6.15 in \cite{MR2906817} or Lemma 3.8 in \cite{MR0453826}. 

\begin{lemma}\label{induct}
Consider a subalgebra $\fa$ of $\fg$. The annihilator ideal of the $\fg$-module $\ind^\fg_\fa M$ for an $\fa$-module $M$ satisfies
\[\Ann_{ U(\fg)}\ind^\fg_\fa M=\Ann_{ U(\fg)}\left(  U(\fg)/\left( U(\fg)\Ann_{ U(\fa)}M\right) \right).\]
In particular, for two $\fa$-modules $M_1$ and $M_2$ it follows that
\[\Ann_{ U(\fa)}M_1\subseteq \Ann_{ U(\fa)}M_2\quad\mbox{implies}\quad \Ann_{ U(\fg)}\ind^{\fg}_{\fa}M_1\subseteq \Ann_{ U(\fg)}\ind^{\fg}_{\fa}M_2.\]
\end{lemma}

An immediate consequence for classical Lie superalgebras is that annihilator ideals of simple finite dimensional modules can never coincide.

\begin{corollary}
\label{lemintdom}
If $\Lambda\in\cP^+$ and $\mu\in\fh_{\oa}^\ast$ are such that $J(\Lambda)=J(\mu)$, then $\Lambda=\mu$.
\end{corollary}

\begin{proof}
The Gelfand-Kirillov dimension immediately implies that $L(\mu)$ is finite dimensional as well, so also $\mu$ is integral dominant and we write $\Lambda_1=\Lambda$, $\Lambda_2=\mu$.

If $\Lambda_1\not=\Lambda_2$, we can assume without loss of generality that $\Lambda_1 \not< \Lambda_2$. It follows immediately that $\left(\Res L(\Lambda_2)\right)_{\chi_{\Lambda_1}^{\oa}}=0$ while  $\left(\Res L(\Lambda_1)\right)_{\chi_{\Lambda_1}^{\oa}}=L_{\oa}(\Lambda_1)$. Corollary \ref{corIAnn} applied to the Lie algebra $\fg_{\oa}$ therefore yields the inequality $J(\Lambda_1)\cap U(\fg_{\oa})\not= J(\Lambda_2)\cap U(\fg_{\oa})$, which concludes the proof.
\end{proof}

The following results are applications of the ideas for primitive ideals of Lie algebras in Section 15.3 in \cite{MR0721170} or Section 15.3.2 in \cite{MR2906817}. A parabolic subalgebra $\fp$ of $\fg$, see \cite{preprint, MR2906817}, is a subalgebra that contains the Borel subalgebra $\fb=\fh\oplus\fn^+$. We have the corresponding parabolic decomposition
\[\fg=\fu^-\oplus\fl\oplus\fu^+\]
with $\fp=\fl\oplus\fu^+$, where $\fl$ is the Levi subalgebra and $\fu^+$ the radical of $\fp$. Define $\zeta_{\fl}:U(\fg)\to U(\fl)$ as the corresponding partial Harish-Chandra projection with kernel $\mathfrak{u}^-U(\fg)+U(\fg)\mathfrak{u}^+$. The simple $\fl$-module with highest weight $\lambda$ is denoted by $L_{\fl}(\lambda)\cong L(\lambda)^{\fu^+}$, with $M^{\fu^+}$ the $\fl$-module consisting of vectors in the $\fg$-module $M$ annihilated by all elements of $\fu^+$, and the corresponding primitive ideal by $I_{\fl}(\lambda)=\Ann_{U(\fl)}L_{\fl}(\lambda)$.

\begin{lemma}
\label{Ianuuu}
The primitive ideal $J(\lambda)$ for $\lambda\in\fh_{\oa}^\ast$ can be expressed as
\[J(\lambda)=\{u\in U(\fg)| \zeta_{\fl}(u_1uu_2)\in I_{\fl}(\lambda)\mbox{ for all }u_1,u_2\in U(\fg)\}.\]
\end{lemma}
\begin{proof}
Since the augmented ideal $U(\fu^-)_+=\fu^-U(\fu^-)$ is stable under the adjoint $\fl$-action, $L(\lambda)^{\fu^+}\cap \fu^-L(\lambda)=0$ follows.
Therefore we have the decomposition of $\fl$-modules
\[L(\lambda)\,\,=\,\, L(\lambda)^{\fu^+}\,\oplus \,\fu^- L(\lambda).\]
We denote the $\fl$-invariant projection of $L(\lambda)$ onto $L(\lambda)^{\fu^+}\cong L_{\fl}(\lambda)$, orthogonal to $\fu^- L(\lambda)$, by $\mP$.

If $u\in J(\lambda)$, then, in particular, $\mP(uv)=0$ for every $v\in L(\lambda)^{\fu^+}$. The decomposition above implies $\mP(uv)=\zeta_{\fl}(u)v$, so $\zeta_{\fl}(u)\in I_{\fl}(\lambda)$. Since $J(\lambda)$ is an ideal, $u_1uu_2$ is also in $J(\lambda)$ and we obtain \[J(\lambda)\subset\{u\in U(\fg)| \zeta_{\fl}(u_1uu_2)\in I_{\fl}(\lambda)\mbox{ for all }u_1,u_2\in U(\fg)\}.\]

Now assume that for an $u\in U(\fg)$ we have that $\zeta_{\fl}(u_1uu_2)\in I_{\fl}(\lambda)\mbox{ for all }u_1,u_2\in U(\fg)$ but $u\not \in J(\lambda)$. Then there are nonzero $x,y\in L(\lambda)$ such that $y=ux$. Since $L(\lambda)$ is a simple $\fg$-module, we can write $x=u_2x'$ and $y'=u_1y$ for some $u_1,u_2\in U(\fg)$ and nonzero $x',y'\in L(\lambda)^{\fu^+}$, which gives $y'=u_1uu_2 x'$. This leads to a contradiction with $\zeta_{\fl}(u_1uu_2)\in I_{\fl}(\lambda)$.
\end{proof}

\begin{corollary}
\label{corIan}
Consider $\fg$ a classical Lie superalgebra with parabolic subalgebra $\fp=\fl+\fu^+$. If $I_{\fl}(\lambda)\subset I_{\fl}(\mu)$, then $J(\lambda)\subset J(\mu)$.
\end{corollary}
\begin{proof}
According to Lemma \ref{Ianuuu}, $u\in J(\lambda)$ implies that $\zeta_{\fl}(u_1uu_2)\in I_{\fl}(\lambda)\mbox{ for all }u_1,u_2\in U(\fg)$, so, in particular, $\zeta_{\fl}(u_1uu_2)\in I_{\fl}(\mu)\mbox{ for all }u_1,u_2\in U(\fg)$. Again by Lemma \ref{Ianuuu}, this implies that $u\in J(\mu)$.
\end{proof}

This implies that Corollary 5.14  in \cite{MR0721170} on simple reflections for Lie algebras can be generalised to reflections for even simple roots for Lie superalgebras.
\begin{lemma}
\label{primsl2}
Consider $\fg$ in list \eqref{list}, but not of queer type. If $\alpha$ simple in $\Delta_{\oa}^+$ and either $\alpha$ or $\alpha/2$ is a simple even root in $\Delta^+$, then
\begin{itemize}
\item $J(s_\alpha\cdot\lambda)=J(\lambda)$ if $\langle \lambda,\alpha^\vee\rangle\not\in\mZ$,
\item $J(s_\alpha\cdot\lambda)\subset J(\lambda)$ if $\langle \lambda,\alpha^\vee\rangle\in\mZ$ and $s_\alpha\cdot\lambda < \lambda$.
\end{itemize}
\end{lemma}
\begin{proof}
Since $\alpha$ (respectively $\alpha/2$) is simple we can define a parabolic subalgebra $\fb\oplus\fg_{-\alpha}$ (respectively $\fb\oplus\fg_{-\alpha/2}\oplus \fg_{-\alpha}$). The Levi subalgebra is given by $\mathfrak{sl}(2)$ (respectively $\mathfrak{osp}(1|2)$), so the result follows from the combination of Corollary \ref{corIan} and the result for $\mathfrak{sl}(2)$ or $\mathfrak{osp}(1|2)$, see Proposition \ref{osp12}.
\end{proof}


\section{Twisting functors on category $\cO$}
\label{sectwist}

For each simple root $\alpha$ in $\Delta^+_{\oa}$ we define an endofunctor $T_\alpha$ on $\cO$. This extends, similarly as in \cite{CMW}, the concept of the corresponding functors for semisimple Lie algebras, originally defined
by Arkhipov in \cite{MR2074588} and further investigated in more detail in \cite{MR1985191, MR2032059, MR2115448, MR2331754}.

For $\alpha\in\Pi_{\oa}$, fix a nonzero $Y\in\left(\mathfrak{g_{\oa}}\right)_{-\alpha}$ and let $U'_{\alpha}$ be the (Ore) localisation of $U(\g)$ with respect to powers of $Y$. Then $U(\g)$ is a subalgebra of the associative algebra $U'_{\alpha}$ and the quotient $U_{\alpha}:=U'_{\alpha}/U(\g)$ has the induced structure of a $U(\g)\text{-}U(\g)$--bimodule. Let $\varphi=\varphi_{\alpha}$ be an automorphism of $\mathfrak{g}$ that maps $\left(\mathfrak{g}_{i}\right)_{\beta}$ to $\left(\mathfrak{g}_{i}\right)_{s_{\alpha}(\beta)}$ for all $\beta\in\Pi$ and $i\in\{\oa,\ob\}$, which can be constructed as in Definition 2.3.4 in \cite{MR2074588}. Finally, consider the bimodule ${}^{\varphi}\U_{\alpha}$, which is obtained from $U_{\alpha}$ by twisting the left action of $U(\g)$ by $\varphi$. The same construction applied to the underlying Lie algebra $\fg_{\oa}$ leads to the $U(\fg_{\oa})\text{-} U(\fg_{\oa})$ bimodule ${}^{\varphi}\U_\alpha^{\oa}$. 

The functors on $\fg$-smod and $\fg_{\oa}$-mod corresponding to tensoring with these bimodules are denoted by $T_\alpha={}^{\varphi}\U_\alpha\otimes_{ U(\fg)}-$ and $T^{\oa}_\alpha={}^{\varphi}\U_\alpha^{\oa}\otimes_{ U(\fg_{\oa})}-$. The latter is the {\em twisting functor} on category $\cO_{\oa}$ from \cite{MR1985191, MR2032059, MR2074588, MR2115448}. The following lemma proves that $T_{\alpha}$ also defines an endofunctor of $\mc O$, which we also call the {\em twisting functor}.

\begin{lemma}
\label{resindT}
The functor $T_\alpha$ restricts to an endofunctor of $\mc O$. The induction and restriction operators intertwine the twisting functors (up to isomorphism of functors):
\[T_\alpha \circ \Ind\cong \Ind\circ T^{\oa}_\alpha\qquad\mbox{and}\qquad \Res\circ T_\alpha\cong 
T_\alpha^{\oa}\circ \Res.\]
\end{lemma}

\begin{proof}
The isomorphism of $\fg\times \fg_{\oa}$-modules $U_\alpha\cong U(\fg)\otimes _{U(\fg_{\oa})} U^{\oa}_\alpha$ implies an isomorphism of functors as follows: $$U_\alpha\otimes_{U(\fg)}U(\fg)\otimes_{U(\fg_{\oa})}{}_-\cong U_\alpha\otimes_{U(\fg_{\oa})}{}_-\cong U(\fg)\otimes_{U(\fg_{\oa})}U^{\oa}_\alpha\otimes_{U(\fg_{\oa})}{}_-,$$ which proves that $T_\alpha \circ \Ind \cong \Ind\circ  T_\alpha^{\oa} $ as functors on $\cO_{\oa}$. The proof of the second isomorphism $ \Res\circ T_\alpha\cong T_\alpha^{\oa}\circ \Res$ is identical to the one of Lemma 3 in \cite{MR2115448}. The fact that $T_\alpha$ takes elements of $\cO$ to elements of $\cO$ then follows from this second property, the definition of $\cO$ as the category of $\fg$-modules which gets mapped by $\Res$ to $\cO_{\oa}$ and the fact that the functor $T_{\oa}^\alpha$ is an endofunctor of $\cO_{\oa}$.
\end{proof}

An interesting case is when $\alpha$ is a simple root of $\Delta^{+}_{\oa}$ and $\alpha/2\in \Delta_{\ob}^+$. Then one could also define a twisting functor by using the localisation of $U(\fg)$ with respect to a non-zero element in $Z\in (\fg_{\ob})_{-\alpha/2}$. For the algebras in the list \eqref{list} this can only occur for $\mathfrak{osp}(m|2n)$ with $m$ odd and for $G(3)$. For those cases $[Z,Z]$ gives a non-zero element of $(\fg_{\oa})_{\alpha}$. The following Lemma then shows that this approach yields an isomorphic twisting functor.

\begin{lemma}
For $\alpha$ simple in $\Delta_{\oa}^+$ and $\alpha/2\in\Delta_{\ob}^+$, the Ore localisation of $ U(\fg)$ with respect to the powers of a nonzero $Y\in (\fg_{\oa})_{-\alpha}$ is isomorphic, as an $ U(\fg)\text{-} U(\fg)$-bimodule, to the Ore localisation of $ U(\fg)$ with respect to the powers of a nonzero $Z\in (\fg_{\ob})_{-\alpha/2}$.
\end{lemma}

\begin{proof}
This follows immediately from the fact that up to a non-zero multiplicative constant we have the relation $Z^2=Y$ in $U(\fg)$.
\end{proof}

As in the classical case, twisting functors satisfy braid relations, in particular, we have the following statement.

\begin{lemma}
\label{lembraid}
Consider $w\in W$ with two reduced expressions $w=s_{\alpha_1}\cdots s_{\alpha_k}$ and $w=s_{\beta_1}\cdots s_{\beta_k}$, then we have
\[T_{\alpha_1}\circ T_{\alpha_2}\circ\cdots\circ T_{\alpha_k}\cong T_{\beta_1}\circ T_{\beta_2}\circ\cdots\circ T_{\beta_k}.\]
\end{lemma}

\begin{proof}
Mutatis mutandis Corollary 11  (including Lemma 13 and Theorem 11) in \cite{MR2115448}.
\end{proof}

\begin{lemma}
\label{Tfreefinite}
The functor $T_\alpha$ is right exact. The left derived functor $\mathcal{L}T_{\alpha}:\cD^b(\mc O)\to\cD^b(\mc O)$ satisfies $\mathcal{L}_iT_{\alpha}=0$ for $i>1$.

Furthermore, for $M\in\cO$ we have the properties
\[\begin{cases}T_\alpha M =0,&\mbox{ if $M$ is $\alpha$-finite;}\\
\cL_1 T_\alpha M=0, &\mbox{ if $M$ is $\alpha$-free.}\end{cases}\]
\end{lemma}
\begin{proof}
The right exactness is immediate from construction. 

The property $\cL_iT_{\alpha}^{\oa}=0$ for $i>1$ is proved in Theorem 2.2 in \cite{MR2032059}. Also the properties on $\alpha$-free and $\alpha$-finite modules for $T_{\alpha}^{\oa}$ follow from \cite{MR2032059}. 
Now we look at the right exact functor \[\Res\circ T_\alpha=T_\alpha^{\oa}\circ \Res: \cO\to \cO_{\oa}.\] Since $\Res$ is exact, $T_\alpha$ sends projective modules to $\Res$-acyclic modules. Conversely, since projective modules in $\cO$ are $\alpha$-free, $\Res$ sends projective modules to $T_\alpha^{\oa}$-acyclic modules (using the properties of $T_{\alpha}^{\oa}$ stated earlier). The Grothendieck spectral sequence (see Section 5.8 in \cite{MR1269324}) therefore yields
\begin{equation}\label{intderfun}\Res\circ \cL_iT_\alpha\,=\,\cL_iT_\alpha^{\oa}\circ \Res\end{equation}
for $i\in\mN$. All properties in the lemma follow from this observation and the corresponding properties of $T_\alpha^{\oa}$.
\end{proof}

\begin{lemma}
\label{charVerma}
Consider $\alpha$ simple in $\Delta_{\oa}^+$ and $\lambda\in\fh_{\oa}^\ast$, then $\ch T_\alpha M(\lambda)=\ch M(s_\alpha\cdot\lambda)$.
\end{lemma}

\begin{proof}
The PBW theorem implies that the $\fg_{\oa}$-module $\Res M(\lambda)$ has a filtration by Verma modules for $\fg_{\oa}$, where the set of occurring highest weights is given by
\[\{\lambda +\sum_{\beta\in I}\beta | I\subset \Delta^-_{\ob}\}.\]
If $\fg$ is of $Q$-type, the weights in the set above appear with a certain constant multiplicity, otherwise exactly once. 

Lemma \ref{Tfreefinite} for $\fg_{\oa}$ and Lemma \ref{resindT} then imply that $\ch T_\alpha M(\lambda)$ is equal to the sum of the characters of $T_\alpha^{\oa}$ acting on the Verma modules in the filtration. Lemma 6.2 in \cite{MR1985191} shows that $\ch T_\alpha^{\oa} M_{\oa}(\mu)=\ch M_{\oa}(s_\alpha\circ\mu)$ for any $\mu\in\fh^\ast_{\oa}$. The character of $T_\alpha M(\lambda)$ is therefore equal to the sum of the characters of Verma modules of $\fg_{\oa}$ with highest weights 
\[\{s_{\alpha}\circ(\lambda +\sum_{\beta\in I}\beta) | I\subset \Delta^-_{\ob}\}=\{s_{\alpha}\cdot\lambda +\sum_{\beta\in I}\beta | I\subset \Delta^-_{\ob}\},\]
which is a standard equality, see e.g. Section 0.5 in \cite{MR1479886}. This proves the lemma.
\end{proof}

\begin{corollary}
\label{sl2osp12q2}
Let $\fg$ be a Lie superalgebra in the set $\{\mathfrak{sl}(2), \mathfrak{osp}(1|2),\mathfrak{q}(2)\}$. Consider $\lambda\in\fh_{\oa}^\ast$, which is assumed to be strongly typical in case $\fg=\mathfrak{q}(2)$, then
\[T_\alpha M(\lambda)=M(s_\alpha\cdot \lambda)\quad\mbox{ if } \langle\lambda,\alpha^\vee+\rho\rangle \ge 0\mbox{ or } \langle\lambda,\alpha^\vee\rangle\not\in\mZ.\]
\end{corollary}
\begin{proof}
In all of these cases the module $M(s_\alpha\cdot \lambda)$ is simple, see \cite{MR2742017} for the case $\fg=\mathfrak{q}(2)$. Since the simple modules are determined (up to possible disregarded parity) by their character, the property follows from Lemma \ref{charVerma}.\end{proof}
 
 \begin{lemma}
\label{TVerma}
Consider $\alpha$ simple in $\Delta_{\oa}^+$ and $\lambda\in\fh_{\oa}^\ast$. Assume that either\begin{itemize}
\item $\alpha\in\Pi$ or $\alpha/2\in \Pi$ with $\left(\fg_{\ob}\right)_{\alpha}=0$ or 
\item $\lambda$ is typical.
\end{itemize}
Then $T_\alpha M(\lambda)=M(s_\alpha\cdot \lambda)$ if $\langle\lambda+\rho,\alpha^\vee\rangle\ge 0$ or $ \langle\lambda+\rho,\alpha^\vee\rangle\not\in\mZ$.
\end{lemma}
Before proving this lemma we note that in the remaining case, namely $\langle\lambda+\rho,\alpha^\vee\rangle \in -\mN$, the module $T_\alpha M(\lambda)$ corresponds to a twisted (or shuffled) Verma module, see \cite{MR1985191}. In this paper we do not study such modules.

\begin{proof}
First we assume that $\fg$ is not of $Q$-type. 

If $\alpha$ or $\alpha/2$ is a simple root, then Verma modules can be defined through a two-step parabolic induction, as in the equation below. If $\alpha$ is simple, the parabolic subalgebra is $\fp_{\alpha}:=\fb+\fg_{-\alpha}$, if $\alpha/2$ is simple, then $\fp_{\alpha}:=\fb+\fg_{-\alpha/2}+\fg_{-\alpha}$. The Verma module is then given by
\[M^{(\fb)}(\lambda)\cong U(\fg)\otimes_{ U(\fp_{\alpha})} U(\fp_{\alpha})\otimes_{ U(\fb)}L_{\fb}(\lambda).\]
The structure of the module $T_\alpha M(\lambda)$ then follows from the corresponding properties for $\mathfrak{sl}(2)$ or $\mathfrak{osp}(1|2)$ (the Levi algebra of the two types of parabolic subalgebras) in Corollary \ref{sl2osp12q2} as in Section 6.6 in \cite{MR1985191}. 

The situation where $\alpha$ is simple in $\Delta_{\oa}^+$ but not in $\Delta^+$ excludes $\fg$ to be in the family of strange algebras $\mathfrak{p}(n)$, so we are left with basic classical Lie superalgebras. Therefore we choose a Borel subalgebra $\hat{\fb}$ of $\fg$ containing $\fb_{\oa}$ in which $\alpha$ or $\alpha/2$ is a simple root. Since $\lambda$ is typical, Corollary \ref{oddtypical} implies that $M^{(\fb)}(\lambda)=M^{(\hat{\fb})}({\lambda+\rho-\hat{\rho}})$. Using the previous result then yields $T_{s_\alpha}M^{(\fb)}(\lambda)=M^{(\hat{\fb})}({s_{\alpha}(\lambda+\rho)-\hat{\rho}})$, from which the claim follows.

Now let $\fg$ be of $Q$-type. Since $\Delta_{\oa}^+=\Delta_{\ob}^+$, the condition $(\fg_{\ob})_\alpha=0$ is never satisfied. The statement regarding strongly typical highest weights follows from Lemma \ref{resindT} and the fact that the equivalence of categories in \cite{Frisk} maps Verma modules to Verma modules if $\lambda$ is regular. The general case follows from reducing to Corollary \ref{sl2osp12q2} for $\mathfrak{q}(2)$, through parabolic induction as for $\mathfrak{sl}(2)$ and $\mathfrak{osp}(1|2)$. 
\end{proof}

\begin{lemma}
\label{simpnonint}
Consider $\lambda\in\fh^\ast_{\oa}$ with $\langle \lambda,\alpha^\vee\rangle\not\in\mZ$. Then $T_\alpha L(\lambda)$ is simple and $T_\alpha^2L(\lambda)=L(\lambda)$.
\end{lemma}

\begin{proof}
Theorem 2.1 in \cite{CMW} implies $T^{\oa}_\alpha L_{\oa}(\nu)=L_{\oa}(s_\alpha\circ\nu)$ if $\langle \nu,\alpha^\vee\rangle\not\in\mZ$. Every simple $\fg_{\oa}$-module in the Jordan-H\"older decomposition of $\Res L(\lambda)$ is of the form $L_{\oa}(\nu)$ with $\langle \nu,\alpha^\vee\rangle\not\in\mZ$ and therefore $\alpha$-free. The left cohomology functors of $T_\alpha^{\oa}$ therefore annihilate any modules composed with these modules $L_{\oa}(\nu)$, according to Lemma \ref{Tfreefinite} applied to $\fg_{\oa}$. Lemma \ref{resindT} then implies that $\Res T_\alpha L(\lambda)$ has a filtration with subquotients $\{L_{\oa}(s_\alpha\circ\nu)\}$, if $\Res L(\lambda)$ has a filtration with subquotients $\{L_{\oa}(\nu)\}$.

Now assume that the module $T_\alpha L(\lambda)$ has a Jordan-H\"older decomposition with subquotients $\{L(\lambda_i)|i=1,\cdots,k\}$, where $k>0$ and $\langle \lambda_i,\alpha^\vee\rangle\not\in\mZ$ because of the previous paragraph. Then each $T_\alpha L(\lambda_i)$ has a Jordan-H\"older decomposition with subquotients $\{L(\lambda_{ij})|j=1,\cdots,k_i\}$, with $k_i>0$. The previous considerations imply \[\ch L(\lambda)=\sum_{i=1}^k\sum_{j=1}^{i_k}\ch L(\lambda_{ij})\] and therefore $k=1$ with $T_\alpha L(\lambda)=L(\lambda_1)$ and $T_\alpha L(\lambda_1)=L(\lambda)$.
\end{proof}

\begin{lemma}
\label{tensorfd}
For $V$ a simple finite dimensional $\fg$-module, the endofunctors of $\cO$ given by $T_\alpha\circ (-\otimes V)$ and $(-\otimes V)\circ T_\alpha$ are isomorphic. The same property holds for the functors $\cL_i T_\alpha$.
\end{lemma}
\begin{proof}
Consider an arbitrary $M\in\cO$. Since the Ore localisation of the beginning of this section only concerns even elements, the proof of this fact can be completed by considerations for Lie algebras made in \cite{MR2032059}, therefore we only sketch it. The comultiplication of the Hopf superalgebra $ U(\fg)$ can be extended to $U_\alpha'$, by the exact same formula as in Lemma 3.1 in \cite{MR2032059}. As in the beginning of the proof of Theorem 3.2 in \cite{MR2032059}, this comultiplication yields an $U(\fg)$-isomorphism
\[U_\alpha'\otimes_{U(\fg)}\left(M\otimes V\right)\cong \left( U_\alpha'\otimes_{U(\fg)}M\right)\otimes V. \]
By factoring out the submodule
\[U(\fg)\otimes_{U(\fg)}\left(M\otimes V\right)\cong M\otimes V\cong \left( U(\fg)\otimes_{U(\fg)}M\right)\otimes V\]
on both sides and twisting the action of $\fg$ by $\varphi$, one obtains
\[{}^{\varphi}U_\alpha\otimes_{U(\fg)}\left(M\otimes V\right)\cong \left( {}^\varphi U_\alpha\otimes_{U(\fg)}M\right)\otimes {}^\varphi V. \]
Since twisting the action of $\fg_{\oa}$ by (the restriction of) $\varphi$ on a simple finite dimensional $\fg_{\oa}$-module yields an isomorphic module (which follows from the central character), we have that $\ch{}^{\varphi}V=\ch V$. Since $V$ was assumed to be simple, we have $ {}^\varphi V\cong V$, so $T_\alpha(M\otimes V)=T_\alpha M\otimes V$. The naturality of the construction then yields an isomorphism between $T_\alpha\circ (-\otimes V)$ and $(-\otimes V)\circ T_\alpha$.

Since the functors $(-\otimes V)$ are exact functors, which map projective modules to projective modules, the isomorphism for the cohomology functors follows from the Grothendieck spectral sequence, see Section 5.8 in \cite{MR1269324}.
\end{proof}

Denote by $S_{\alpha}^{\overline{0}}$ (respectively $S_\alpha$) the subfunctor of the identity functor on $\mathcal{O}_{\overline{0}}$ (respectively $\cO$) given by taking the maximal 
$\alpha$-finite submodule. Correspondingly we denote by $Z_{\alpha}^{\overline{0}}$ (respectively $Z_\alpha$) the subfunctor of the identity functor on $\mathcal{O}_{\overline{0}}$ (respectively $\cO$) given by taking the maximal 
$\alpha$-finite quotient. These functors also intertwine the functors $\Res$ and $\Ind$ and commute as functors with taking tensor products with finite dimensional representations. 

\begin{proposition}
\label{LTZuck}
We have the following isomorphisms of endofunctors on $\cO$;
$$\cL_1T_\alpha\cong S_\alpha\qquad\mbox{and}\qquad \cR_1 G_\alpha\cong Z_\alpha. $$
\end{proposition}
\begin{proof}
We use the equivalence of categories of Gorelik in \cite{MR1862800} (restricted to category~$\cO$) between $\cO_\chi$ and $(\cO_{\oa})_{\widetilde\chi}$ for a strongly typical central character $\chi$ and a matching character $\widetilde\chi$ for the underlying Lie algebra. The functors for this equivalence are given in equation \eqref{eqGorelik}.
Based on equation \eqref{intderfun}, we find that on $\cO_\chi$
$$\cL_1T_\alpha\cong\left(\Ind-\right)_{\chi}\circ \cL_1T^{\oa}_\alpha\circ \left(\Res -\right)_{\widetilde\chi},$$
where $S_\alpha$ is in the same way generated by $S_\alpha^{\oa}$. On a strongly typical block the equivalence therefore follows from the classical case in Theorem 2(2) in \cite{MR2331754}.

In particular, this implies that $\cL_1T_\alpha$ and $S_\alpha$ coincide on the full additive subcategory of $\cO$ generated by one integral dominant strongly typical injective module. Every indecomposable injective module in $\cO$ is  a direct summand of the tensor product of a strongly typical integral dominant injective module and a finite dimensional representation, which is a consequence of the fact that the set of strongly typical weights is closed in the Zariski topology, while the set of integral dominant ones is dense. Since both $\cL_1 T_\alpha$ and $S_\alpha$ commute as functors with the functors corresponding to taking the tensor product with finite dimensional modules, see Lemma \ref{tensorfd}, it follows that the two functors are isomorphic on the full subcategory of injective modules in $\cO$. Since both are left exact, this suffices to conclude that they are equivalent on $\cO$. The proof for $G_\alpha$ is completely analogous, using projective modules.
\end{proof}

We have the following generalisation of Corollary 4.2 in \cite{MR2032059} and Theorem 2.1 in \cite{CMW}.

\begin{proposition}
\label{autoequi}
For any central character $\chi: \mathcal{Z}(\fg)\to\mC$, the left derived functor $\mathcal{L}T_{\alpha}$ is an auto-equivalence of the bounded derived category $\mathcal{D}^b(\mc O_\chi)$.
\end{proposition}

\begin{proof}
To prove that $T_\alpha$ preserves every block $\cO_\chi$, it suffices to prove that the action induced on the Grothendieck group $K(\cO)$ preserves the blocks $K(\cO_\chi)$. Since the Verma modules in $\cO_\chi$ induce a basis of the Grothendieck group $K(\cO_\chi)$, the property follows from Lemma \ref{charVerma}. The left derived functor $\cL T_\alpha$ therefore defines an endofunctor of $\mathcal{D}^b(\mc O_\chi)$. 

Decomposition \eqref{decompnonint} implies a decomposition of categories $\cO_\chi=\cO_\chi^{\alpha\,{\rm i.}}\oplus \cO_\chi^{\alpha\,{\rm n. i.}}$, where the simple modules $L(\mu)$ in $\cO_\chi^{\alpha\,{\rm i.}}$ satisfy $\langle \mu,\alpha^\vee\rangle\in\mZ$ and the simple modules $L(\mu)$ in $\cO_\chi^{\alpha\,{\rm n.i.}}$ satisfy $\langle \mu,\alpha^\vee\rangle\not\in\mZ$.

Lemma \ref{simpnonint} implies that $T_\alpha$ yields a bijection of the simple modules in $\cO_\chi^{\alpha\,{\rm n. i.}}$, while Proposition~\ref{LTZuck} implies that $T_\alpha$ is exact. Similarly to the proof of Theorem 2.1 of \cite{CMW}
one then shows that $T_\alpha$ induces an self-equivalence of $\cO_\chi^{\alpha\,{\rm n. i.}}$.

Now we prove that $\cL T_\alpha$ is an auto-equivalence of $\cD^b(\cO_\chi^{\alpha\,{\rm i.}})$, where $\cO_\chi^{\alpha\,{\rm i.}}$ decomposes into blocks $\cO_\lambda$ with $\langle \lambda,\alpha^\vee\rangle\in\mZ$. 
We denote by $G_\alpha$ the (right) adjoint functor to $T_\alpha$ on $\cO$, which by definition is left-exact. 
Similarly, denote by $G_\alpha^{\overline{0}}$ the (right) adjoint functor to $T_\alpha^{\overline{0}}$ on $\cO_{\overline{0}}$. Applying the adjunction automorphism to Lemma~\ref{resindT} we get that 
\begin{equation}\label{eqeq1}
G_\alpha \circ \Ind\cong \Ind\circ G^{\oa}_\alpha\qquad\mbox{and}\qquad \Res\circ G_\alpha\cong 
G_\alpha^{\oa}\circ \Res.
\end{equation}

For a projective module $P\in \cO$, consider the adjunction morphism $P\to G_\alpha \circ T_\alpha P$. 
Applying $\Res$ and using Lemma~\ref{resindT}, isomorphisms \eqref{eqeq1}, the fact that $\Res P$ is projective in 
$\cO_{\overline{0}}$ and Corollary 4.2 in \cite{MR2032059}, we get that 
$\Res P\to \Res G_\alpha\circ T_\alpha P\cong G_\alpha^{\oa}\circ T_\alpha^{\oa}\,\Res P$ 
is an isomorphism and thus we also get that $P\to G_\alpha\circ T_\alpha P$ is
an isomorphism. Similarly one shows that for any 
injective module $I\in \cO$ the adjunction morphism $T_\alpha\circ G_\alpha I\to I$ is an isomorphism.

Similarly to Theorem 5.9 in \cite{MR2139933} it follows that $\mathcal{R}G_\alpha\circ \mathcal{L}T_\alpha$
is isomorphic to the identity functor on $\mathcal{D}^-(\cO_\chi^{\alpha\,{\rm i.}})$ and that 
$\mathcal{L}T_\alpha\circ\mathcal{R}G_\alpha $
is isomorphic to the identity functor on $\mathcal{D}^+(\cO_\chi^{\alpha\,{\rm i.}})$.
Since $\mathcal{D}^b(\cO_\chi^{\alpha\,{\rm i.}})$ is the intersection of 
$\mathcal{D}^-(\cO_\chi^{\alpha\,{\rm i.}})$ and $\mathcal{D}^+(\cO_\chi^{\alpha\,{\rm i.}})$,
the claim follows.
\end{proof}

As a consequence of Proposition~\ref{autoequi}, we can generalise the equivalence of categories in Proposition 3.9 of \cite{CMW}. Since we can describe this result more precisely using the star action, we postpone this  statement to the subsequent Proposition \ref{equivcatni}.

Now we prove how the twisting functor and its left cohomology functors behave with respect to arbitrary simple modules in $\cO$.

\begin{theorem}
\label{Tsimple}
Consider $\lambda\in \fh_{\oa}^\ast$ and $\alpha$ simple in $\Delta_{\oa}^+$.\\
(i) If $L(\lambda)$ is $\alpha$-finite, then
\[T_\alpha L(\lambda)=0 \quad\qquad \mbox{and}\qquad\quad \cL_1T_\alpha L(\lambda)\cong L(\lambda).\]
(ii) If $L(\lambda)$ is $\alpha$-free with $\langle\lambda,\alpha^\vee\rangle\in\mZ$, then
\begin{itemize}
\item $\cL_1T_\alpha L(\lambda)=0$,
\item $T_\alpha L(\lambda)$ has simple top isomorphic to $L(\lambda)$,
\item the kernel of $T_\alpha L(\lambda)\tto L(\lambda)$ is an $\alpha$-finite module,
\item $\dim\Hom_{\cO}\left(L(\mu),T_\alpha L(\lambda)\right)=\dim\Ext_{\cO}^1(L(\mu),L(\lambda))$ for all $L(\mu)$ which are $\alpha$-finite.
\end{itemize}
\noindent
(iii) If $\langle \lambda,\alpha^\vee\rangle\not\in\mZ$, then $\cL_1T_\alpha L(\lambda)=0$ and $T_\alpha L(\lambda)$ is a simple module.
\end{theorem}

\begin{proof}
First consider $L(\lambda) $ to be $\alpha$-finite. The first property is a part of Lemma \ref{Tfreefinite}. The second property follows from Proposition~\ref{LTZuck}.

Now consider $L(\lambda) $ to be $\alpha$-free with $\langle \lambda,\alpha^\vee\rangle\in\mZ$. First we study $\ch T_\alpha L(\lambda)$. For this we can restrict to the reductive Lie algebra $\fa=\fh_{\oa}\oplus (\fg_{\oa})_{\alpha}\oplus(\fg_{\oa})_{-\alpha}$, which is isomorphic to the direct sum of its centre with $\mathfrak{sl}(2)$. The combination of Lemma 3 in \cite{MR2115448} and Lemma~\ref{resindT} implies $\res^{\fg}_{\fa}T_\alpha L(\lambda)=T^{\fa}_{\alpha} \res^{\fg}_{\fa}L(\lambda)$, with $T^{\fa}_{\alpha}$ the twisting functor for $\mathfrak{sl}(2)$. Since $L(\lambda)$ is a simple (and therefore self-dual in $\cO$) $\alpha$-free module, it follows that $\res^{\fg}_{\fa} L(\lambda)$ is a tilting module for $\mathfrak{sl}(2)$. Lemma 4 in \cite{MR2115448} therefore yields $\ch T_\alpha L(\lambda)= \ch L(\lambda)\oplus \ch N$ with $N$ an $\alpha$-free module. This implies that the only $\alpha$-free simple subquotient of  $T_\alpha L(\lambda)$ is $L(\lambda)$ (occurring with multiplicity one). The remainder of the proof is identical to the proof of Theorem 6.3 in \cite{MR2032059}, by using Proposition \ref{autoequi}.

Finally, consider $\langle \lambda,\alpha^\vee\rangle\not\in\mZ$, then $\cL_1T_\alpha L(\lambda)=0$ follows from Lemma \ref{Tfreefinite} and $T_\alpha L(\lambda)$ is simple by Lemma \ref{simpnonint}.
\end{proof}

\begin{remark}{\rm
The highest weight of the simple module $T_\alpha L(\lambda)$ if $\langle \lambda,\alpha^\vee\rangle\not\in\mZ$ is given by the star action of $s_\alpha$ on $\lambda$, as will be developed in Section \ref{secstar}. If $\alpha$ or $\alpha/2$ is simple in $\Delta^+$ and $(\fg_{\ob})_\alpha=0$, then $T_\alpha L(\lambda)=L(s_\alpha\cdot \lambda)$.}
\end{remark}

\begin{corollary}
For $\lambda\in\fh^\ast_{\oa}$ with $\langle \lambda,\alpha^\vee\rangle\in\mZ$ and $L(\lambda)$ $\alpha$-free we have the following isomorphisms:
\[\Soc \,T_\alpha L(\lambda)\,\,\cong \,\, \Soc \,\Rad\, T_\alpha L(\lambda)\,\,\cong\,\, \Top\, \Rad\, T_\alpha L(\lambda).\]
\end{corollary}
\begin{proof}
The first isomorphism is trivial. Theorem \ref{Tsimple}(ii) implies there is a short exact sequence $\Rad T_\alpha L(\lambda)\hookrightarrow T_\alpha L(\lambda) \tto L(\lambda)$. For an arbitrary $L(\mu)$, which is $\alpha$-finite, we apply the left exact contravariant functor $\Hom_{\cO}(-,L(\mu))$ to this exact sequence, yielding a long exact sequence. The results of Theorem \ref{Tsimple}(ii) imply that this exact sequence reduces to
\[0\to \Hom_{\cO}(\Rad T_\alpha L(\lambda), L(\mu))\to \Ext_{\cO}^1( L(\lambda),L(\mu))\to \Ext_{\cO}^1(T_\alpha L(\lambda),L(\mu))\to\cdots .\]
By Theorem \ref{Tsimple}(ii) we have $\dim\Ext^1_{\cO}(L(\lambda),L(\mu))=\dim\Hom_{\cO}(L(\mu),T_\alpha L(\lambda))$, so it suffices to prove $ \Ext_{\cO}^1(T_\alpha L(\lambda),L(\mu))=0$. According to Proposition \ref{autoequi}, Theorem \ref{Tsimple}(i) and using a reasoning as in the proof of Theorem \ref{Tsimple}, it follows that this first extension is isomorphic to $\Hom_{\cO}(L(\lambda),L(\mu))=0$.
\end{proof}

Now we study how twisting functors are related to annihilator ideals.

\begin{lemma}
\label{inclannT}
If $L(\lambda)$ is $\alpha$-free, we have $\Ann_{ U(\fg)} T_\alpha L(\lambda)=\Ann_{ U(\fg)}L(\lambda)$. If $L(\lambda)$ is $\alpha$-finite, $\Ann_{ U(\fg)} T_\alpha L(\lambda)=U(\fg)$.
\end{lemma}

\begin{proof}
When $L(\lambda)$ is $\alpha$-finite, the result follows from Theorem \ref{Tsimple}(i). Thus from now on we assume that $L(\lambda)$ is $\alpha$-free.

The set $\cY$ of powers of $Y\in\fg_{-\alpha}$ used for the Ore localisation then satisfies $\cY\cap J(\lambda)=\emptyset$. Therefore $J(\lambda) \,U'_{\alpha}=U'_{\alpha}\,J(\lambda)$ is a proper ideal in $U'_\alpha$. This yields $$J(\lambda)U'_\alpha\otimes_{U(\fg)}L(\lambda)=U'_\alpha\otimes_{U(\fg)}J(\lambda)L(\lambda)=0,$$ so $J(\lambda)\subseteq\Ann_{U(\fg)}U'_{\alpha}\otimes_{U(\fg)}L(\lambda)$. Since $U_\alpha\otimes_{U(\fg)}L(\lambda)$ is a quotient of $U'_\alpha\otimes_{U(\fg)}L(\lambda)$, we have
\begin{equation}\label{helpannT}\Ann_{U(\fg)}T_\alpha L(\lambda)\supseteq\varphi(J(\lambda)).\end{equation}
  
Now if $\langle \lambda,\alpha^\vee\rangle\in\mZ$, Theorem \ref{Tsimple}(ii) implies that $[T_\alpha L(\lambda):L(\lambda)]=1$, so we have the inclusion $\Ann_{U(\fg)}T_\alpha L(\lambda)\subset J(\lambda)$. Equation \eqref{helpannT} thus implies $\varphi(J(\lambda))\subset J(\lambda)$. The automorphism $\varphi^{-1}$ of $\fg$ also satisfies the properties required at the beginning of this section. The twisting functor could thus also be defined using that automorphism, leading to the conclusion $\varphi^{-1}(J(\lambda))\subset J(\lambda)$, which yields $\varphi(J(\lambda))=J(\lambda)$.

Now if $\langle \lambda,\alpha^\vee\rangle\not\in\mZ$, Proposition \ref{simpnonint} implies that $T_\alpha L(\lambda)=L(\mu)$ for some $\mu\in\fh_{\oa}^\ast$ with $T_\alpha L(\mu)=\lambda$. Equation \eqref{helpannT} therefore yields $\varphi(J(\lambda))\subset J(\mu)$ and $\varphi(J(\mu))\subset J(\lambda)$. The result then follows again from considering $\varphi^{-1}$.
\end{proof}

\begin{lemma}
\label{projT}
If $L(\lambda)$ is $\alpha$-free with $\langle \lambda,\alpha^\vee\rangle\in\mZ$, then $T_\alpha P(\lambda)\cong
P(\lambda)$.
\end{lemma}

\begin{proof}
First we prove the result for $\fg_{\oa}$. For the particular case of regular integral weights, this was already proved in Proposition 5.3 in \cite{MR2032059}. Let $\mu\in\fh_{\oa}^\ast$ be a $W$-maximal weight satisfying $s_\alpha\cdot\mu=\mu$, then $M_{\oa}(\mu)=P_{\oa}(\mu)$ is projective in $\cO_{\oa}$. Lemma \ref{TVerma} then implies $T_\alpha P_{\oa}(\mu)=P_{\oa}(\mu)$. This result generalises to arbitrary indecomposable projective with $\alpha$-free top by tensoring with finite dimensional modules and using Lemma \ref{tensorfd}. 

It is checked by a standard argument that $\Res P(\lambda)$ decomposes into projective modules in $\cO_{\oa}$ with $\alpha$-free top. From this, the result above and Lemma~\ref{resindT} it follows that $T_\alpha P(\lambda)$ and $P(\lambda)$ have the same characters. 

Now we determine the top of the module $T_\alpha P(\lambda)$. We denote by $G_\alpha$ the right adjoint functor to $T_\alpha$. Below we will twice use that Theorem 4.1 in \cite{MR2032059} and Lemma \ref{resindT} imply that $\ch G_\alpha V=\ch T_\alpha V$ for $V$ a simple module in $\cO$. Consider a simple $\alpha$-free module $L(\mu)$, Theorem \ref{Tsimple}(ii) implies that the only simple $\alpha$-finite subquotient of $G_\alpha L(\mu)$ is $L(\mu)$, appearing with multiplicity one. So we have
$$\Hom_{\cO}(T_\alpha P(\lambda), L(\mu))=\Hom_{\cO}(P(\lambda),G_\alpha L(\mu))=[G_\alpha L(\mu): L(\lambda)]=\delta_{\lambda\mu}.$$
Now consider $L(\nu)$ to be $\alpha$-finite. Theorem \ref{Tsimple}(i) then implies $G_\alpha L(\nu)=0$, so
$$\Hom_{\cO}(T_\alpha P(\lambda), L(\nu))=\Hom_{\cO}(P(\lambda),G_\alpha L(\nu))=0.$$

It follows that 
$T_\alpha P(\lambda)$ has simple top $L(\lambda)$ and hence $T_\alpha P(\lambda)$ is a quotient of $P(\lambda)$.
Since the characters of $T_\alpha P(\lambda)$ and $P(\lambda)$ coincide, we get $T_\alpha P(\lambda)=P(\lambda)$.
\end{proof}

An important tool to establish equalities between primitive ideals is provided in the following lemma.

\begin{lemma}
\label{extann}
If weights $\lambda,\mu\in\fh_{\oa}^\ast$ and a root $\alpha$ (simple in $\Delta_{\oa}^+$) satisfy the following properties:
\begin{itemize}
\item $\langle \lambda,\alpha^\vee\rangle\in\mZ$ and $\langle \mu,\alpha^\vee\rangle\in\mZ$,
\item $L(\lambda)$ is $\alpha$-finite and $L (\mu)$ is $\alpha$-free and,
\item $\Ext^1_{\cO}(L(\lambda),L(\mu))\not=0$;
\end{itemize}
then we have the relation $\Ann_{ U(\fg)} L(\mu)\subseteq\Ann_{ U(\fg)} L(\lambda)$.
\end{lemma}

\begin{proof}
Theorem \ref{Tsimple} implies that $\Hom_{\cO}(L(\lambda),T_{\alpha}L(\mu))\not=0$. The result then follows from Lemma \ref{inclannT}.
\end{proof}


\section{Primitive ideals for typical blocks}
\label{sectyp}

In this section we study inclusions between primitive ideals corresponding to typical highest weights for classical Lie superalgebras of type I and basic classical Lie superalgebras of type II, in the list \eqref{list}. The corresponding statements for Lie superalgebras of queer type are slightly less general and are presented in Section \ref{secqn}.

\begin{theorem}
\label{typeItyp}
Let $\fg$ be a classical Lie superalgebra of type I with distinguished system of positive roots. Consider $\lambda,\mu\in\fh^\ast$ with $\lambda$ or $\mu$ typical, then
\begin{eqnarray*}
J(\lambda)\subseteq J(\mu) &\Leftrightarrow & I(\lambda)\subseteq I(\mu).
\end{eqnarray*} 
\end{theorem}

\begin{proof}
If $\fg$ is basic classical, the result follows immediately from the subsequent Theorem \ref{typeIItyp}, since the notion of strongly typical coincides with typical for $\fg$ of type I and the fact that a perfect mate of $\chi_\lambda$ is given by $\chi^0_\lambda$, see \cite{MR1862800}. Therefore we only need to consider the strange Lie superalgebras.

If $\fg$ is of the strange type $\mathfrak{p}(n)$ or $\tilde{\mathfrak{p}}(n)$, then $U(\fg)$ either has no centre or a centre that acts by the same central character on each simple representation. After factoring out the Jacobson radical, the concept of typical central characters does arise. The central characters separate in particular between typical highest weights in different orbits and separate typical from atypical highest weights, see Lemma 5.1 in \cite{MR1943937}. We denote the Jacobson radical by $J$ and the quotient map by $q: U(\fg)\to U(\fg)/J$. If $J(\lambda)\subseteq J(\mu)$, then clearly $q(J(\lambda))\subseteq q(J(\mu))$. The typicality of one of the two weights therefore forces the other one the be in the same orbit of the Weyl group.

For typical weights simple modules are induced, namely, $L(\lambda)=U(\fg)\otimes_{U(\fg_0+\fg_1)}L_0(\lambda)$, see Theorem~5.2 and Lemma 5.4 in \cite{MR1943937}. Lemma \ref{induct} therefore implies the ``if'' part. On the other hand, $L_0(\lambda)=\left(\res_{\fg_0}^{\fg}L(\lambda)\right)_{\chi^0_\lambda}$, see again Theorem 5.2 in \cite{MR1943937}. Lemma \ref{IplusAnn} applied to $\fg_0$ therefore implies the ``only if'' part.
\end{proof}

\begin{theorem}
\label{typeIItyp}
Let $\fg$ be a basic classical Lie superalgebra. Consider a strongly typical central character $\chi: \mathcal{Z}(\fg)\to \mC$ with perfect mate $\widetilde{\chi}: \mathcal{Z}(\fg_{\oa})\to \mC$. Denote by $\lambda$ the $W$-maximal (for $\rho$-shifted action) weight that satisfies $\chi_\lambda=\chi$ and accordingly by $\widetilde{\lambda}$ the $W$-maximal (for $\rho_{\oa}$-shifted action) weight corresponding to $\widetilde{\chi}$. Then
\[J(w\cdot\lambda)\subseteq J(\mu) \quad\mbox{or}\quad J(\nu)\subseteq J(w\cdot \lambda)\]
for $w\in W$ and $\nu,\mu\in\fh_{\oa}^\ast$ if and only if there are $w_1,w_2\in W$ such that $\mu=w_1\cdot\lambda$, $\nu=w_2\cdot \lambda$ and
\[I(w\circ\widetilde{\lambda})\subseteq I(w_1\circ \widetilde{\lambda})\quad \mbox{and}\quad I(w_2\circ\widetilde{\lambda})\subseteq I(w\circ \widetilde{\lambda}),\,\, \mbox{ respectively}.\]

Consider a (not necessarily strongly) typical (regular) dominant weight $\Lambda$, then for $w,w'\in W$ we have $$J(w\cdot\Lambda)\subseteq J(w'\cdot\Lambda)\Leftrightarrow I(w\circ\Lambda)\subseteq I(w'\circ\Lambda).$$
\end{theorem}
\begin{proof}
First we consider $\lambda$ strongly typical. In order to have an inclusion between primitive ideals they have to admit the same central character. The typicality of one of the highest weights therefore forces the other one to be typical as well. Moreover, the typicality of $\lambda$ forces $\mu$ and $\nu$ to be in the orbit of $\lambda$. According to Theorem 1.3.1 in \cite{MR1862800}, the equalities
\[\left(\Ind L_{\oa}(w\circ\widetilde{\lambda})\right)_{\chi}= L(w\cdot \lambda)\quad\mbox{and}\quad \left(\Res L(w\cdot\lambda)\right)_{\widetilde{\chi}}= L_{\oa}(w\circ \widetilde{\lambda})\]
hold for every $w\in W$. The statements then follow by application of Lemmata \ref{IplusAnn} and \ref{induct}.

Theorem 1.3 and Theorem 1.4 of  \cite{MR1479886} imply that the posets of primitive ideals of two regular typical blocks, for which the difference of the dominant weights in the orbits is in $\cP$, are isomorphic. The results for regular strongly typical characters therefore carry over to regular typical characters.
\end{proof}

\begin{remark}{\rm
Theorem \ref{typeIItyp} for regular typical central characters rederives Theorem B in \cite{MR1479886} for $\mathfrak{osp}(1|2n)$. Theorem 1.4 in \cite{MR1479886} describes the poset of primitive ideals corresponding to a singular typical central character in terms of that of a regular typical central character for basic classical Lie superalgebras. The combination of Theorem 1.4 in \cite{MR1479886} and Theorem \ref{typeIItyp} therefore also yields all inclusions between primitive ideals for singular typical characters for basic classical Lie superalgebras.}
\end{remark}


\section{Generic weights}
\label{secgen}

As in \cite{MR1309652,MR1201236}, we call weights which are far from the walls of the Weyl chambers generic. How far weights have to be from the wall for our purposes is determined by the odd roots, therefore we define the sets
\begin{equation}
\label{gammaset}
\Gamma=\{\sum_{\alpha\in I}\alpha\,|I\,\subset \Delta^-_{\ob}\}\quad\mbox{and}\quad \widetilde{\Gamma}=\{\sum_{\alpha\in I}\alpha\,|\,I\subset \Delta_{\ob}\}.
\end{equation}
Since we might have $\sum_{\alpha\in I}\alpha=\sum_{\alpha\in I'}\alpha$ even when $I\not= I'$, we need to stress that we interpret these sets with {\em multiplicities}.

In the following, the notion Weyl chambers refers to the Weyl chambers of the $\rho_{\oa}$-shifted action for the whole Weyl group $W$ (even when non-integral weights are considered).

\begin{definition}
\label{defgeneric}
(i) We call a weight $\lambda\in\fh_{\oa}^\ast$ {\em weakly generic} if all weights in the set $\lambda+\widetilde{\Gamma}$ are inside the same Weyl chamber.

(ii)We call a weight $\lambda\in\fh_{\oa}^\ast$ {\em generic} if each weight in the set $\lambda+\Gamma$ is weakly generic.

(iii)We call a simple highest weight module $L(\mu)$ {\em generic} if all highest weights $\nu$ of the simple $\fg_{\oa}$-subquotients of $\Res L(\mu)$ are weakly generic.
\end{definition}
Since the set $\widetilde{\Gamma}$ is invariant under the Weyl group (which is a consequence of the fact that $\Lambda \fg_{\ob}$ is a finite dimensional $\fg_{\oa}$-module), a weight $\lambda$ is weakly generic if and only if $w\circ\lambda$ is weakly generic for an arbitrary $w\in W$. Furthermore, as in Section 0.5 in \cite{MR1479886}, we have the equality of sets $$w\circ(\lambda+{\Gamma})=w\cdot \lambda +{\Gamma},$$
which implies that a weight $\lambda$ is generic if and only if $w\cdot\lambda$ is generic for an arbitrary $w\in W$.

\begin{lemma}
\label{farcr}
If $\lambda\in\fh^\ast_{\oa}$ is weakly generic, then $\Res L(\lambda)$ is completely reducible and
\[\Res L(\lambda)\subset  \bigoplus_{\gamma\in\Gamma}\left(L_{\oa}(\lambda+\gamma)\right)^{\oplus k}\]
for some $k\in\mN$.
\end{lemma}

\begin{proof}
If $\fg$ is of type I, this follows from the morphism $\Lambda\fg_{-1}\otimes L_0(\lambda)\tto \Res L(\lambda)$. If $\fg$ is basic classical of type II, this follows easily as in the first part of the proof of Lemma 3.4 in \cite{MR1479886}. Thus we are left to consider the algebras of $Q$-type, the following proof covers all cases however.

We consider the $\fg_{\oa}$-module $U(\fg_{\oa})[L(\lambda)]_\lambda$, where $[M]_\lambda$ stands for the space of vectors with weight $\lambda$ in a weight module $M$. This $\fg_{\oa}$-submodule of $\Res L(\lambda)$ has a module in its socle of the form $L_{\oa}(w\circ\lambda)$ for some $w\in W$. Thus we have $L_{\oa}(w\circ\lambda)\hookrightarrow \Res L(\lambda)$.

From the Frobenius reciprocity it then follows that $ \Ind L_{\oa}(w\circ\lambda)\tto L(\lambda)$ and therefore we have \[\Lambda\fg_{\ob}\otimes L_{\oa}(w\circ \lambda)\tto \Res L(\lambda).\] The weights of the highest weight vectors appearing in the left hand side have to be in the set $w\circ \lambda+\widetilde{\Gamma}=w\circ (\lambda+\widetilde{\Gamma})$. By assumption, all these weights are in the same chamber, so $\Lambda \fg_{\ob}\otimes L_{\oa}(w\circ\lambda)$ is completely reducible, which implies that $\Res L(\lambda)$ is completely reducible and $w=1$. 

By the standard filtration of $\Res M(\lambda)$ in $\cO_{\oa}$ and the fact that all appearing weights are in the same chamber, 
for some $k\in\mN$ we have
\[\Res M(\lambda)=\bigoplus_{\gamma\in \Gamma} \left(M_{\oa}(\lambda+\gamma)\right)^{\oplus k}.\]
Since this projects onto the completely reducible module $\Res L(\lambda)$, the result follows.
\end{proof}

\begin{corollary}
\label{corgen}
If $\lambda\in\fh^\ast_{\oa}$ is generic, then $L(\lambda)$ is a generic module. For any generic simple module $M$ in $\cO$, the module $\Res M$ is completely reducible.
\end{corollary}

\begin{lemma}
\label{farneq}
Consider two weakly generic dominant weights $\Lambda_1,\Lambda_2$ and $w_1,w_2\in W$. Then the inclusion $J(w_1\cdot\Lambda_1)\subseteq J(w_2\cdot \Lambda_2)$ implies $\Lambda_1-\Lambda_2\in\cP$ and $I(w_1\circ\Lambda_1)\subseteq I(w_2\circ\Lambda_1)$.
\end{lemma}

\begin{proof}
The property $U(\fg_{\oa})\cap J(w_1\cdot\Lambda_1)\subseteq U(\fg_{\oa})\cap J(w_2\cdot\Lambda_2)$ is equivalent to \[\Ann_{U(\fg_{\oa})}\Res L(w_1\cdot\Lambda_1)\subseteq \Ann_{U(\fg_{\oa})}\Res L(w_2\cdot\Lambda_2).\] 
Lemma \ref{farcr} and Lemma \ref{IplusAnn} applied to $\fg_{\oa}$ imply that
\begin{itemize}
\item $w_2w_1^{-1}\circ (w_1\cdot\Lambda_1)\in w_2\cdot\Lambda_2+\Gamma$ with
\item $I(w_1\cdot \Lambda_1)\subseteq I(w_2w_1^{-1}\circ (w_1\cdot\Lambda_1) )$.
\end{itemize}

The first property can be rewritten as $\Lambda_1-\Lambda_2\in w_2^{-1}(\rho_{\ob})-w_1^{-1}(\rho_{\ob})+w_2^{-1}(\Gamma)$ which, in particular, implies $\Lambda_1-\Lambda_2\in\cP$.

The second property can be rewritten as $I(w_1\circ\Lambda)\subseteq I(w_2\circ \Lambda)$ with $\Lambda=\Lambda_1+w_1^{-1}(\rho_{\ob})-\rho_{\ob}$. Here $\Lambda$ is a dominant weight in $\Lambda_1+\cP$, so Lemma \ref{translclass} implies that $I(w_1\circ\Lambda)\subseteq I(w_2\circ \Lambda)$ is equivalent to $I(w_1\circ\Lambda_1)\subseteq I(w_2\circ \Lambda_1)$.
\end{proof}

\begin{lemma}
\label{farneq2}
Consider two weakly generic weights $\lambda,\mu$ in the same chamber. The equality $J(\lambda)=J(\mu)$ implies $\lambda=\mu$.
\end{lemma}
\begin{proof}
Mutatis mutandis Corollary \ref{lemintdom}. \end{proof}

\begin{lemma}
\label{twistinggeneric}
If $L(\lambda)$ is a generic module, then all subquotients in the Jordan-H\"older series of $T_\alpha L(\lambda)$ are generic.
\end{lemma}

\begin{proof}
If $L_{\oa}(\mu)$ is a subquotient in the Jordan-H\"older series of the $\fg_{\oa}$-module \[\Res T_\alpha L(\lambda)= T^{\oa}_\alpha \Res L(\lambda),\] 
the fact that $T^{\oa}_\alpha$ preserves blocks of $\cO_{\oa}$ implies that $\mu$ is weakly generic.
\end{proof}

By definition, subtracting or adding distinct odd roots to weakly generic weights yields regular weights. Interesting properties of generic weights therefore stem from the following observations, which will be applied in later sections.

\begin{lemma}
\label{regatyp1}
Consider $\fg$ in the list \eqref{list} excluding $\mathfrak{p}(n)$ and $\widetilde{\mathfrak{p}}(n)$. If a weight $\mu\in\fh_{\oa}^\ast$ is regular, then it is not atypical with respect to both of two roots $\gamma,\gamma'$ which satisfy $\langle \gamma,\gamma'\rangle\not=0$.
\end{lemma}

\begin{proof}
First consider $\fg$ to be basic classical, then $\gamma,\gamma'$ are isotropic. For such isotropic roots either $\gamma+\gamma'$ or $\gamma-\gamma'$ is an element of $\Delta_{\oa}$ (this follows immediately from the structure of the roots, see Chapter 2 in \cite{MR2906817} for $\mathfrak{gl}(m|n)$ and $\mathfrak{osp}(m|2n)$ and Section 1.4 for $D(2,1;\alpha)$, $F(4)$ and $G(3)$). The equalities $\langle \mu +\rho,\gamma\rangle=0=\langle \mu+\rho,\gamma'\rangle$ would therefore imply that $\langle \mu+\rho,\alpha\rangle =0$ for an $\alpha\in \Delta_{\oa}$ which contradicts the regularity of $\mu$.

Now consider $\fg$ of $Q$-type (so $\rho=0$). The claim amounts to proving that for a regular weight it is not possible to have both $\langle\lambda,\epsilon_i+\epsilon_j\rangle$ and $\langle\lambda,\epsilon_i+\epsilon_l\rangle$ equal to zero if $j\not=l$. This follows immediately since otherwise we would have $\langle\lambda,\epsilon_j-\epsilon_l\rangle=0$.
\end{proof}

\begin{lemma}
\label{noextraatyp}
Consider $\fg$ in the list \eqref{list} excluding $\mathfrak{p}(n)$ and $\widetilde{\mathfrak{p}}(n)$ and a regular weight $\lambda\in\fh_{\oa}^\ast$ which is atypical with respect to a root $\gamma$. If $\lambda+\gamma$ (respectively $\lambda-\gamma$) is also regular, then $\lambda+\gamma$ (respectively $\lambda-\gamma$) has exactly the same atypical roots as $\lambda$.
\end{lemma}
\begin{proof}
First consider $\fg$ to be basic classical and $\lambda+\gamma$ regular. Assume there is an isotopic root $\gamma'$ such that $\langle \lambda+\gamma+\rho,\gamma'\rangle=0$ with $\langle \lambda+\rho,\gamma'\rangle\not=0$. This immediately implies $\langle \gamma,\gamma'\rangle\not=0$. Lemma \ref{regatyp1} therefore implies that $\langle \lambda+\gamma+\rho,\gamma'\rangle\not=0$ since $\lambda+\gamma$ is regular and $\langle \lambda+\gamma+\rho,\gamma\rangle=0$. That the combination $\langle \lambda+\gamma+\rho,\gamma'\rangle\not=0$ with $\langle \lambda+\rho,\gamma'\rangle=0$ is impossible follows similarly. The case $\lambda-\gamma$ follows from the first case by substituting $\lambda\to \lambda+\gamma$ and $\lambda-\gamma\to\lambda$.

The proof for $Q$-type is similar. 
\end{proof}

\section{Star actions and deformed Weyl group orbits}
\label{secstar}

In \cite{Dimitar}, Gorelik and Grantcharov introduced a deformation of the usual action of the Weyl group for $\mathfrak{q}(n)$ called the star action. As is already clear from Proposition \ref{Primq2} and Lemma \ref{helpqncor}, this star action is closely related to inclusions between primitive ideals. We will prove that this action is also naturally linked to the twisting functors. First we introduce an analogue of this star action for basic classical Lie superalgebras. Even though the definition is entirely different, these star action are also closely related to primitive ideals and twisting functors. 

In Subsection \ref{starbasic} we define star actions for basic classical Lie superalgebras and derive important properties. In Subsection \ref{starqn} we show that the corresponding properties also hold for the star action of $\mathfrak{q}(n)$. The easiest example of a star action for $\mathfrak{osp}(m|2n)$ is given in Subsection \ref{exstarosp}, the easiest example for $\mathfrak{gl}(m|n)$ corresponds to the usual $\rho$-shifted action of the Weyl group.  

\subsection{Basic classical Lie superalgebras, the general case}
\label{starbasic}

For a Borel subalgebra $\fb_{\oa}$ of $\fg_{\oa}$ we define $B(\fb_{\oa})$ as the set of Borel subalgebras for which the underlying Borel subalgebra of $\fg_{\oa}$ is equal to $\fb_{\oa}$. Throughout this entire section we assume $\fb_{\oa}$ as fixed. Corresponding to $\fb_{\oa}$ we have the basis $\Pi_{\oa}$ of simple roots in $\Delta_{\oa}^+$.

A {\em star action map} is an arbitrary mapping 
\[\cB:\Pi_{\oa}\to B(\fb_{\oa}),\quad\alpha\mapsto\cB(\alpha),\] 
with the condition that $\alpha$ or $\alpha/2$ is simple in the system of positive roots corresponding to the Borel subalgebra $\cB(\alpha)$. For each star action map, we will define a star action, similar to the one for $\mathfrak{q}(n)$ in \cite{Dimitar}. As in \cite{Dimitar}, those star actions do not lead to an action of the Weyl group, but to an action of an infinite Coxeter group, which we denote by $\widetilde{W}$. This group is freely generated by $\{s_\alpha\}$ for all $\alpha$ simple in $\Delta_{\oa}^+$ subject to the relation $s_\alpha^2=1$. Contrary to the Coxeter group for the star action in \cite{Dimitar}, we cannot impose the relation $s_\alpha s_\beta=s_\beta s_\alpha$ if $\langle \alpha,\beta\rangle=0$. An example of this follows immediately from the star action defined in the beginning of Section \ref{primosp}.

\begin{definition}
\label{defStargen}
Let $\fg$ be a basic classical Lie superalgebra with a Borel subalgebra $\fb$. For each star action map $\cB$ we define the \textbf{star action} $\ast_\cB:\widetilde{W}\times \fh^\ast\to\fh^\ast$. For $\alpha$ simple in $\Delta^+_{\oa}$ and $\lambda\in\fh^\ast$ we denote $\widetilde{\fb}=\cB(\alpha)$ and define $\widetilde{\lambda}$ by the relation $L^{(\fb)}(\lambda)\cong L^{(\widetilde{\fb})}(\widetilde{\lambda})$. The weight $s_\alpha\ast_{\cB}\lambda$ is then defined by the relation
\[L^{(\fb)}(s_\alpha\ast_{\cB} \lambda)\cong L^{(\widetilde{\fb})}(s_\alpha(\widetilde{\lambda}+\widetilde{\rho})-\widetilde{\rho}).\]
If we want to specify the reference Borel subalgebra, the star map is denoted by $\ast^{\fb}_{\cB}$.
\end{definition}

\begin{remark}{\rm
\label{starchar}
Both from construction and from the combination of Lemma \ref{oddrefl} with the description of central characters through the Harish-Chandra isomorphism, it follows that we have $\chi_{\lambda}=\chi_{s_{\alpha}\ast\lambda}$.}
\end{remark}

The highest weights with respect to different systems of positive roots can be calculated through the technique of odd reflections, see Lemma \ref{oddrefl}. The property $s_{\alpha}\ast_{\cB} s_{\alpha}\ast_{\cB}\lambda=\lambda$ follows easily for any $\lambda\in\fh^\ast$, so the star action is well defined as an action of $\widetilde{W}$. We will use the short notation $s_{\alpha_1}\cdots s_{\alpha_{k-1}} s_{\alpha_k}\ast\lambda$ for $s_{\alpha_1}\ast\cdots\ast s_{\alpha_{k-1}}\ast s_{\alpha_k}\ast\lambda$.
In case $\lambda$ is typical, Corollary \ref{oddtypical} implies that this definition yields \[s_{\alpha}\ast_{\cB}\lambda=s_\alpha\cdot\lambda=s_{\alpha}(\lambda+\rho)-\rho.\] So, in particular, the action on typical weights does not depend on the choice of $\cB$.

A trivial example of $\ast_{\cB}$ is given by choosing $\fg$ to be a basic classical Lie superalgebra of type~I, $\fb$ the distinguished Borel subalgebra and $\cB(\alpha):=\fb$, for every even simple root $\alpha$. This leads to the usual $\rho$-shifted action.

The connection of such star actions with primitive ideals and twisting functors  follows from the following observations.

\begin{lemma}
\label{starSTHW}
For an $\alpha$-free $L(\mu)$ and for any star action mapping $\cB$, we have the relation
\[[T_\alpha L(\mu): L(s_\alpha\ast_{\cB} \mu)]\not=0.\]
\end{lemma}
\begin{proof}
We denote $\widetilde{\fb}=\cB(\alpha)$. If $L^{(\fb)}(\mu)=L^{(\widetilde{\fb})}(\widetilde{\mu})$ is $\alpha$-free, then again a parabolic reduction to $\mathfrak{sl}(2)$ or $\mathfrak{osp}(1|2)$ shows that the highest weight (in the relevant system of positive roots) in the module $T_\alpha L^{(\widetilde{\fb})}(\widetilde{\mu})$ is $s_\alpha(\widetilde{\mu}+\widetilde{\rho})-\widetilde{\rho}$. This means that
 \[[T_\alpha L(\mu): L(s_{\alpha}\ast_{\cB} \mu)]=[T_\alpha L^{(\widetilde{\fb})}(\widetilde{\mu}): L^{(\widetilde{\fb})}(s_\alpha(\widetilde{\mu}+\widetilde{\rho})-\widetilde{\rho})]\not=0,\]
 which proves the lemma.
\end{proof}

Thanks to the star action we can now give a complete analogue of Corollary 5.14 in \cite{MR0721170} for basic classical Lie superalgebras. 

\begin{corollary}
\label{corastS}
Consider $\lambda\in\fh^\ast$ and an arbitrary star action $\ast_{\cB}$. We have
\begin{itemize}
\item if $\langle \alpha^\vee,\lambda\rangle \not\in\mZ$, then $J(s_\alpha\ast_{\cB}\lambda)=J(\lambda)$;
\item if $\langle \alpha^\vee,\lambda\rangle \in\mZ$ with $L(\lambda)$ $\alpha$-finite, then $J(s_\alpha\ast_{\cB}\lambda)\subseteq J(\lambda)$.
\end{itemize}
\end{corollary}

\begin{proof}
We use the notation of Definition \ref{defStargen}. We have 
\[J(\lambda)=\Ann_{U(\fg)}L^{(\hat{\fb})}(\hat{\lambda})\quad\mbox{and}\quad J(s_{\alpha}\ast_{\cB}\lambda)=\Ann_{U(\fg)}L^{(\hat{\fb})}(s_\alpha(\hat{\lambda}+\hat{\rho})-\hat{\rho}).\]
The result then follows from Lemma \ref{primsl2}.
\end{proof}

\begin{corollary}
\label{corast2}
If $\langle\alpha^\vee,\lambda\rangle\not\in\mZ$, then $s_{\alpha}\ast_{\cB}\lambda$ does not depend on the choice of star action map $\cB$, moreover $L(s_{\cB}\ast\lambda)\,=\,T_\alpha L(\lambda)$.
\end{corollary}

\begin{proof}
This is an immediate consequence of Lemmata \ref{starSTHW} and \ref{simpnonint}.
\end{proof}

This corollary now allows us to describe the generalisation of Proposition 3.9 in \cite{CMW} explicitly. We write $s_\alpha\ast\mu$ for $s_\alpha\ast_{\cB}\mu$ if $\langle \mu,\alpha^\vee\rangle\not\in\mZ$ since there is no dependence on the star action map $\cB$.

\begin{proposition}
\label{equivcatni}
Consider $\fg$ a basic classical Lie superalgebra, $\lambda \in \fh^*$ and a positive root $\alpha$ simple in $\Delta_{\oa}^+$ such that
$\langle \lambda, \alpha^\vee \rangle \not\in \mZ$. Then the functor $T_\alpha: \cO_\lambda \rightarrow \cO_{s_\alpha \ast \lambda}=\cO_{s_\alpha\cdot\lambda}$ is
an equivalence of categories sending $L(\mu)$ to $L(s_\alpha \ast \mu)$ for each $L(\mu)\in\cO_\lambda$.

If $\alpha$ (or $\alpha/2$) is simple in $\Delta^+$ as well, then $T_\alpha$ also maps $M(\mu)$ to $M(s_\alpha\ast\mu)=M(s_\alpha\cdot\mu)$.
\end{proposition}

\begin{proof}
The property $\cO_{s_\alpha \ast \lambda}=\cO_{s_\alpha\cdot\lambda}$ follows from Remark \ref{starchar} and the fact that by definition of the star action $s_\alpha\ast\lambda-s_\alpha\cdot\lambda \in \mZ\Delta_{\ob}\subset \cP$.

The first result follows from Corollary \ref{corast2} and the proof of Proposition \ref{autoequi}.

If $\alpha$ is simple in $\Delta^+$, we can choose a star action map $\cB$ for which $\cB(\alpha)=\alpha$, then by definition $s_\alpha\ast_{\cB}\mu=s_\alpha\cdot\mu$. By Corollary \ref{corast2} this equation holds for any star action map if $\langle \mu,\alpha^\vee\rangle\not\in\mZ$. The result then follows from Lemma \ref{TVerma}.
\end{proof}

Now we prove that Corollary \ref{corast2} can be extended from a reflection with respect to a non-integral simple root of $\lambda$, to the entire set of left coset of representatives $W^\lambda$ for $W_\lambda$ in $W$, given in equation \eqref{coset}, as long as only reduced expressions are employed.

\begin{proposition}
\label{cosetWeylprop}
Consider $\lambda\in \fh^\ast$, $w\in W^\lambda$ and two $\underline{reduced}$ expressions $s_{\alpha_1}\cdots s_{\alpha_k}$ and $s_{\beta_1}\cdots s_{\beta_k}$ for $w$. Then  we have $s_{\alpha_1}\cdots s_{\alpha_k}\ast_\cB\lambda=s_{\beta_1}\cdots s_{\beta_k}\ast_\cB\lambda$. Furthermore, $s_{\alpha_1}\cdots s_{\alpha_k}\ast_\cB\lambda$ does not depend on $\cB$.
\end{proposition}

\begin{proof}
The proof of Theorem 15.3.7 in \cite{MR2906817} states that $\langle \alpha_i^\vee,s_{\alpha_{i+1}}\cdots s_{\alpha_k}\lambda\rangle\not\in\mZ$ for any $1\le i\le k$ where the same is true for $\{\beta_i\}$. Iteratively applying Corollary \ref{corast2} then yields
\begin{itemize}
\item $L(s_{\alpha_1}\cdots s_{\alpha_k}\ast_{\cB}\lambda)=T_{\alpha_1}\cdots T_{\alpha_k} L(\lambda)$ and
\item $L(s_{\beta_1}\cdots s_{\beta_k}\ast_{\cB}\lambda)=T_{\beta_1}\cdots T_{\beta_k} L(\lambda)$.
\end{itemize}
The claim then follows from the fact that $T_{\alpha_1}\cdots T_{\alpha_k} =T_{\beta_1}\cdots T_{\beta_k} $, see Lemma \ref{lembraid}.
\end{proof}

In general $s_{\alpha}\ast_\cB\lambda$ actually depends on the choice of $\cB$. We give an easy example.

\begin{example}
\label{exMusson}
{\rm Consider $\mg=\mathfrak{gl}(2|1)$. We choose the distinguished system of positive roots, with simple roots $\epsilon_1-\epsilon_2$ and $\epsilon_2-\delta$. The Weyl group contains only one nontrivial element $s=s_{\epsilon_1-\epsilon_2}$. Star action mappings $\cB$ are therefore determined by $\cB(\epsilon_1-\epsilon_2)$. Thus there are two choices for star actions, the usual dot action and $\ast_\cB$, with $\hat{\fb}=\cB(s)$ the Borel subalgebra corresponding to simple positive roots $\delta-\epsilon_1$ and $\epsilon_1-\epsilon_2$. We consider the atypical regular, but not (weakly) generic weight $\lambda=k\epsilon_1+k\epsilon_2-k\delta$ with $k\in\mZ$. The usual dot action yields 
\[s\cdot\lambda=(k-1)\epsilon_1+(k+1)\epsilon_2-k\delta.\]
Applying the odd reflections yields $\hat{\lambda}=\lambda$, thus $s(\hat{\lambda}+\hat{\rho})-\hat{\rho}=(k-1)\epsilon_1+(k+1)\epsilon_2-k\delta=s\cdot\lambda$. Applying the odd reflections again then yields \[s\ast_{\cB}\lambda=k\epsilon_1+(k+1)\epsilon_2-(k+1)\delta.\]
}
\end{example}

Since one Lie superalgebra can have more than one star action, Corollary \ref{corastS} implies that more inclusions between primitive ideals can be expected compared to the classical case. Based on the result of Letzter in Theorem 3.7 of \cite{MR1362685} for type I Lie superalgebras, these extra inclusions should correspond to inclusions between primitive ideals for highest weights not in the same (usual $\rho$-shifted) orbit of the Weyl group. In particular, Example \ref{exMusson} yields such an extra inclusion for $\mathfrak{sl}(2|1)$.

\begin{remark}{\rm
For $\fg=\mathfrak{sl}(2|1)$, Example \ref{exMusson} and Corollary \ref{corastS} imply the inclusion
\[J(\epsilon_2-\delta)\subseteq J(0).\]
As can be seen from the result of Musson in \cite{MR1231215}, repeated in Proposition \ref{sl21}, this inclusion is the only inclusion between primitive ideals for highest weights in different (usual $\rho$-shifed) orbits for the case $\mathfrak{sl}(2|1)$. This example immediately illustrates the usefulness of the different star actions in the study of primitive ideals.}
\end{remark}

The following theorem shows that such examples can only appear close to the walls of the Weyl chambers since the star actions become identical and regular in the generic region (see Definition \ref{defgeneric}).

\begin{theorem}
\label{Sstarfar}
Consider $\fg$ equal to $\mathfrak{gl}(m|n)$ or $\mathfrak{osp}(m|2n)$ and $\lambda\in\fh^\ast$ a weakly generic weight. We take an arbitrary Borel subalgebra $\fb$ and a star action map $\cB:\Pi_{\oa}\to B(\fb_{\oa})$. Consider $w\in W$ with two expressions $s_{\alpha_1}\cdots s_{\alpha_k}$ and $s_{\beta_1}\cdots s_{\beta_l}$ for $w$. Then we have $$s_{\alpha_1}\cdots s_{\alpha_k}\ast_{\cB}\lambda=s_{\beta_1}\cdots s_{\beta_l}\ast_{\cB}\lambda.$$ The $\ast_{\cB}$-orbit of $\lambda$ is regular in the sense that it intersects each Weyl chamber exactly once. Moreover, $w\ast_{\cB} \lambda$ does not depend on the choice of $\cB$.
\end{theorem}

\begin{proof}
First we prove this for $\fg=\mathfrak{gl}(m|n)$. Consider the case where $\fb$ is the distinguished Borel subalgebra, thus corresponding to the simple roots
\[\epsilon_1-\epsilon_2,\cdots, \epsilon_{m-1}-\epsilon_m,\epsilon_m-\delta_1,\delta_1-\delta_2,\cdots,\delta_{m-1}-\delta_m,\]
and $\cB$ is an arbitrary star action map. We prove that for $\alpha\in\Pi_{\oa}$ and a weakly generic $\lambda\in\fh^\ast$, we have the equality 
\[s_\alpha\ast_{\cB}\lambda=s_\alpha\cdot\lambda=s_\alpha(\lambda+\rho)-\rho=s_{\alpha}(\lambda+\rho_0)-\rho_0.\]
Without loss of generality we may assume $\alpha=\epsilon_i-\epsilon_{i+1}$. Now $\hat{\fb}=\cB(\alpha)$ is an arbitrary Borel subalgebra for which $\alpha$ is a simple root. We have the corresponding ordered set of odd reflections $\{\gamma^{(1)},\cdots,\gamma^{(k)}\}$ linking $\fb$ to $\hat{\fb}$. The combination of Lemma \ref{oddrefl} and Lemma \ref{noextraatyp} implies that 
\begin{equation}\label{eqproof}s_\alpha\ast_{\cB}\lambda=s_{\alpha}(\lambda+\rho+\gamma_1+\cdots+\gamma_p)-\gamma_1'-\cdots-\gamma_q'-\rho\end{equation}
with
\begin{itemize}
\item $\{\gamma_1,\cdots,\gamma_p\}$ the set of all roots in $\{\gamma^{(1)},\cdots,\gamma^{(k)}\}$ which satisfy $\langle \lambda+\rho,\gamma^{(i)}\rangle=0$,
\item $\{\gamma'_1,\cdots,\gamma'_p\}$ the set of all roots in $\{\gamma^{(1)},\cdots,\gamma^{(k)}\}$ which satisfy $\langle \lambda+\rho,s_\alpha(\gamma^{(i)})\rangle=0$.
\end{itemize}
This implies that, if $\langle\alpha,\gamma^{(i)}\rangle=0$, regardless of $\lambda$-atypicality with respect to it, this root drops out of equation \eqref{eqproof}. This means that only isotropic roots of the form $\pm(\epsilon_i -\delta_l)$ or $\pm(\epsilon_{i+1}-\delta_l)$ for some $l$ can be relevant. Since $\epsilon_i-\epsilon_{i+1}$ is simple in $\hat{\fb}$, an odd reflection with respect to $\epsilon_i -\delta_l$ occurs (in passing from $\fb$ to $\hat{\fb}$) if and only if one with respect to $\epsilon_{i+1} -\delta_l$ occurs. Since $\epsilon_{i+1}-\delta_l=s_{\epsilon_i-\epsilon_{i+1}}(\epsilon_i-\delta_l)$, it follows that every atypical root in equation \eqref{eqproof} is cancelled out. Thus we have $s_\alpha\ast_{\cB}\lambda=s_\alpha\cdot\lambda$, which implies that all these star actions are equivalent and satisfy the proposed properties.

Now we consider an arbitrary reference Borel subalgebra ${\fb}$ of $\mathfrak{gl}(m|n)$. We take an arbitrary $\alpha$ simple in $\Delta_{\oa}^+$, star action map $\cB$ with $\cB(\alpha)=\hat{\fb}$ and weakly generic weight $\lambda$. The definition of the star action implies
\[L^{(\fb)}(s_\alpha\ast^{{\fb}}_{\cB}\lambda)=L^{(\hat{\fb})}(s_{\alpha}(\hat{\lambda}+\hat{\rho})-\hat{\rho})).\]
On the other hand, the result for the distinguished Borel subalgebra (which we denote by $\fb^d$ now) with $L^{(\fb)}(\lambda)\cong L^{(\fb^d)}(\lambda^d)\cong L^{(\hat{\fb})}(\hat{\lambda})$, implies
\[L^{(\fb^d)}(s_\alpha(\lambda^d+\rho^d)-\rho^d)\cong L^{(\fb^d)}(s_\alpha\ast^{\fb^d}_{\cB}\ast\lambda^d)\cong L^{(\hat{\fb})}(s_{\alpha}(\hat{\lambda}+\hat{\rho})-\hat{\rho})). \]
So, in particular, we obtain $L^{(\fb)}(s_\alpha\ast^{{\fb}}_{\cB}\lambda)\cong L^{(\fb^d)}(s_\alpha(\lambda^d+\rho^d)-\rho^d)$ which implies that $\ast^{{\fb}}_{\cB}$ is independent of $\cB$ and leads to an action of the Weyl group.

Now we prove the theorem for $\fg=\mathfrak{osp}(m|2n)$. The role of the usual dot action for $\mathfrak{gl}(m|n)$ will now be played by the example $\ast'$ in the subsequent Subsection \ref{exstarosp}. We start by assuming that $\fb$ is the distinguished Borel subalgebra, with corresponding system of positive roots as described in the beginning of Subsection \ref{exstarosp}. We only need to prove that for $\lambda$ generic we have $s_\alpha\ast_\cB\lambda=s_\alpha\ast'\lambda$. The other statements then follow, as for $\mathfrak{gl}(m|n)$, from the subsequent Lemma \ref{starexpreg}.

We use the notation for the roots of Subsection \ref{exstarosp}. If $\alpha$ is of the form $\epsilon_i-\epsilon_{i-1}$ or $\delta_j-\delta_{j+1}$, the proof is identical to that of $\mathfrak{gl}(m|n)$. Now consider $\alpha=2\delta_n$. As for $\mathfrak{gl}(m|n)$ it follows that for generic weights only odd reflections with respect to isotropic roots $\gamma$ that satisfy $\langle \gamma,\alpha\rangle\not=0$ can be relevant, the other ones get cancelled out in expressions as in equation \eqref{eqproof}. Any choice of Borel subalgebra for which $2\delta_n$ or $\delta_n$, depending on $m=2d$ or $m=2d+1$, is simple is obtained from $\hat{\fb}$ in Subsection \ref{exstarosp} by odd reflections of the form $\delta_i-\epsilon_j$ with $i<n$. This implies that the corresponding star action $s_{2\delta_n}\ast_{\cB}$ will lead to the same result as the star action $s_{2\delta_n}\ast'$ in Subsection \ref{exstarosp}. Consider $\alpha=\epsilon_d$ for $m=2d+1$. Since $\alpha$ can never be simple in a system of roots for which an odd reflection with respect to $\delta_j-\epsilon_d$ needs to be used (starting from the distinguished one), the star action of $s_\alpha$ on generic weights becomes the usual $\rho$-shifted action of $\alpha$. Finally, let $\alpha=\epsilon_{d-1}+\epsilon_d$ for $m=2d$. The result then follows as before since in any system of positive roots in which $\epsilon_{d-1}+\epsilon_d$ is simple, $\epsilon_{d-1}-\epsilon_d$ is also simple.
\end{proof}

\subsection{An important example}
\label{exstarosp}
In this section we consider $\fg=\mathfrak{osp}(m|2n)$. We work with the standard choice of positive roots, which leads to simple positive roots
\begin{eqnarray*}
\delta_1-\delta_2,\cdots,\delta_{n-1}-\delta_n, \delta_n-\epsilon_1, \epsilon_1-\epsilon_2,\cdots, \epsilon_{d-1}-\epsilon_d,\epsilon_d\qquad \mbox{for }m=2d+1,\\
\delta_1-\delta_2,\cdots,\delta_{n-1}-\delta_n, \delta_n-\epsilon_1, \epsilon_1-\epsilon_2,\cdots, \epsilon_{d-1}-\epsilon_d,\epsilon_{d-1}+\epsilon_d\qquad \mbox{for }m=2d.
\end{eqnarray*}

The only simple root of $\Delta_{\oa}^+$ which is not simple in this choice is $2\delta_n$. For this root we will use the Borel subalgebra $\hat{\fb}$ corresponding to simple positive roots
\begin{eqnarray*}
\delta_1-\delta_2,\cdots,\delta_{n-2}-\delta_{n-1}, \delta_{n-1}-\epsilon_1, \epsilon_1-\epsilon_2,\cdots, \epsilon_{d-1}-\epsilon_d,\epsilon_d-\delta_n,\delta_n\qquad \mbox{for }m=2d+1,\\
\delta_1-\delta_2,\cdots,\delta_{n-2}-\delta_{n-1}, \delta_{n-1}-\epsilon_1, \epsilon_1-\epsilon_2,\cdots, \epsilon_{d-1}-\epsilon_d,\epsilon_{d}-\delta_n,2\delta_n\qquad \mbox{for }m=2d.
\end{eqnarray*}
The odd reflections that lead $\fb$ to $\hat{\fb}$ are $\{\delta_n-\epsilon_1,\delta_n-\epsilon_2,\cdots,\delta_n-\epsilon_d\}$.

We can then define a star action map $\cB$ for which $\cB(\alpha)=\fb$ if $\alpha\not=2\delta_n$ and $\cB(2\delta_n)=\hat{\fb}$. From now on we shorten the notation of the corresponding star action $\ast^{\fb}_\cB$ for $\mathfrak{osp}(m|2n)$ to $\ast'$.

First we prove that, as for $\mathfrak{q}(n)$ in Statement 4.1.4 in \cite{Dimitar}, this star action for $\mathfrak{osp}(m|2n)$ can be used as a criterion for $\alpha$-finiteness.
\begin{lemma}
\label{starfinosp}
Consider $\fg=\mathfrak{osp}(m|2n)$ with standard Borel subalgebra and $\alpha$ simple in $\Delta_{\oa}^+$. The module $L(\lambda)$ is $\alpha$-finite if and only if $\langle \lambda,\alpha^\vee\rangle\in\mZ$ and $s_\alpha\ast'\lambda <\lambda$.
\end{lemma}
\begin{proof}
By Lemma \ref{lemfinfree}, only the case where $\langle \lambda,\alpha^\vee\rangle\in\mZ$ and $\alpha=2\delta_n$ is nontrivial. Since $\delta_n$ or $2\delta_n$ is simple in the system of positive roots for $\hat{\fb}$, Lemma \ref{lemfinfree} implies that $L^{(\fb)}(\lambda)=L^{(\hat{\fb})}(\hat{\lambda})$ is $\alpha$-finite if and only if $s_{\alpha}(\hat{\lambda}+\hat{\rho})< \hat{\lambda}+\hat{\rho}$ in $\hat{\Delta}^+$. Therefore we show that $s_{\alpha}(\hat{\lambda}+\hat{\rho})< \hat{\lambda}+\hat{\rho}$ is equivalent to $s_\alpha\ast'\lambda <\lambda$ (in the relevant systems of positive roots). We write out explicitly
\[\hat{\lambda}=\sum_{i=1}^nk_i\delta_i+\sum_{j=1}^dl_j\epsilon_j, \quad \lambda=\sum_{i=1}^na_i\delta_i+\sum_{j=1}^db_j\epsilon_j\,\,\mbox{ and}\quad s_\alpha\ast'\lambda  =\sum_{i=1}^nc_i\delta_i+\sum_{j=1}^dd_j\epsilon_j,\]
with $k_n\in\mZ$, since by assumption  $\langle \lambda,\alpha^\vee\rangle\in\mZ$. The condition  $s_{\alpha}(\hat{\lambda}+\hat{\rho})< \hat{\lambda}+\hat{\rho}$ is then equivalent to $k_n>1-k_n$ if $m=2d+1$ and $k_n>2-k_n$ if $m=2d$. It is clear that $k_i=a_i=c_i$ for $i<n$. First we prove that if $s_{\alpha}(\hat{\lambda}+\hat{\rho})< \hat{\lambda}+\hat{\rho}$, then $a_n\ge c_n$. If we would have $a_n < c_n$, then, after applying one of the odd reflections, say $\delta_n-\epsilon_s$, the coefficients of $\delta_n$ in the weights coming from $\hat{\lambda}$ and $s_{\alpha}(\hat{\lambda}+\hat{\rho})-\hat{\rho}$ should coincide. Since the coeffiients by $\{\epsilon_1,\cdots,\epsilon_{s-1}\}$ have not yet been affected, the remaining isotropic roots will be atypical for both weights at the same time, which implies $a_n=c_n$. Furthermore, this procedure shows that if $a_n=c_n$, the coefficients by $\{\delta_1,\cdots,\delta_n,\epsilon_1,\cdots,\epsilon_{s-1}\}$ of $\lambda$ and $s_\alpha\ast'\lambda$ will coincide and that the coefficient of $\epsilon_s$ will be one higher in $\lambda$ than in $s_\alpha\ast'\lambda$, so $\lambda>s_\alpha\ast'\lambda$ follows. If $a_n>c_n$, then $\lambda>s_\alpha\ast'\lambda$ follows immediately. The ``only'' if part follows similarly.
\end{proof}

For weakly generic weights this star action takes a very simple expression, as is shown in the following lemma.
\begin{lemma}
\label{starexexp}
For $\lambda\in\fh^\ast$ a weakly generic weight for $\fg=\mathfrak{osp}(m|2n)$ with $d=\lfloor m/2\rfloor$ we have
\[s_{2\delta_n}\ast'\lambda=\begin{cases}s_{2\delta_n}\cdot\lambda,&\mbox{if } \langle \lambda+\rho,\delta_n\pm \epsilon_{i}\rangle\not=0 \mbox{ for all } 1\le i\le d;\\
s_{2\delta_n}\cdot\lambda -\delta_n-\epsilon_i,&\mbox{if } \langle \lambda+\rho,\delta_n-\epsilon_{i}\rangle=0\mbox{ for some } 1\le i\le d;\\
s_{2\delta_n}\cdot\lambda -\delta_n+\epsilon_i,&\mbox{if } \langle \lambda+\rho,\delta_n+\epsilon_{i}\rangle=0\mbox{ for some } 1\le i\le d.
\end{cases}\]
\end{lemma}
\begin{proof}
Lemma \ref{regatyp1} implies that for a regular weight only one root in the set $\{\delta_n\pm\epsilon_i|i=1,\dots,d\}$ can be atypical. The case where none of the roots are atypical is trivial. 

We consider the case where there is an $i$ for which $\langle \lambda+\rho,\delta_n -\epsilon_i\rangle=0$, the other situation being similar. Lemma \ref{oddrefl} and the fact that $\lambda +\delta_n-\epsilon_i$ has the same atypicalities as $\lambda$, see Lemma \ref{noextraatyp}, then imply that $\hat{\lambda}=\lambda+\delta_n-\epsilon_i+\rho-\hat{\rho}$, so $$s_{2\delta_n}(\hat{\lambda}+\hat{\rho})-\hat{\rho}=s_{2\delta_n}\cdot\lambda\,\,-\,\delta_n-\epsilon_i+\rho-\hat{\rho}.$$

The value of $$\langle s_{2\delta_n}\cdot\lambda-\delta_n-\epsilon_i+\rho,\delta_n-\epsilon_j\rangle=-\langle \lambda+\delta_n-\epsilon_i+\rho,\delta_n+\epsilon_j\rangle$$
can never be zero by Lemma \ref{regatyp1}, since $\lambda+\delta_n-\epsilon_i$ is regular and $\delta_n-\epsilon_i$-atypical. Lemma \ref{oddrefl} therefore implies
\[L^{(\fb)}(s_{2\delta_n}\cdot\lambda -\delta_n-\epsilon_i)=L^{(\hat{\fb})}(s_{2\delta_n}\cdot\lambda -\delta_n-\epsilon_i+\rho-\hat{\rho}),\]
which proves the lemma.
\end{proof}

Now we prove that this star action leads to an action of the Weyl group, which is required for the proof of Theorem \ref{Sstarfar} for the general case.

\begin{lemma}
\label{starexpreg}
Consider $\fg=\mathfrak{osp}(m|2n)$ and $w\in W$ with two expressions $s_{\alpha_1}\cdots s_{\alpha_k}$ and $s_{\beta_1}\cdots s_{\beta_l}$ for $w$. Then $s_{\alpha_1}\cdots s_{\alpha_k}\ast'\lambda=s_{\beta_1}\cdots s_{\beta_l}\ast'\lambda$ for $\lambda$ a weakly generic weight. Furthermore, the star orbit of $\lambda$ is a set of regular weights which intersects each Weyl chamber exactly once.
\end{lemma}
\begin{proof}
First, we focus on the Weyl group $W(\mathfrak{sp}(2n))$ of $\mathfrak{sp}(2n)$. For $i=1,\cdots,n$ we define the weight $\gamma^{(i)}$ as
\begin{itemize}
\item $0$ if $\langle \lambda+\rho,\delta_i\pm\epsilon_j\rangle\not=0$ for all $1\le j\le d$,
\item $\delta_i-\epsilon_l$ if $\langle \lambda+\rho,\delta_i-\epsilon_l\rangle=0$,
\item $\delta_i+\epsilon_l$ if $\langle \lambda+\rho,\delta_i+\epsilon_l\rangle=0$.
\end{itemize} 
We then claim
\[s_{\alpha_1}\cdots s_{\alpha_k}\ast'\lambda= s_{\alpha_1}\cdots s_{\alpha_k}\cdot(\lambda + \gamma^{(i_1)}+\cdots \gamma^{(i_k)})\]
where $\gamma^{(i_j)}$ appears if and only if $w=s_{\alpha_1}\cdots s_{\alpha_k}\in W(\mathfrak{sp}(2n))$ flips the sign of the coefficient which originally was in the $i_j$-th position. This statement can be proved by induction using Lemma \ref{starexexp}, similarly to the proof of Proposition \ref{orbitqnfar}.

All these weights are weakly generic. It follows from Lemma \ref{starexexp} that for weakly generic weights, the star action $\ast'$ of the Weyl group of $\mathfrak{sp}(2n)$ commutes with that of $\mathfrak{so}(m)$. This concludes the proof.
\end{proof}

\subsection{The case $\mathfrak{q}(n)$}
\label{starqn}

We recall the notation $\overline{\epsilon_i-\epsilon_j}=\epsilon_i+\epsilon_j$ for $1\le i <j\le n$. The star action for $\mathfrak{q}(n)$ as defined in Section 4 in \cite{Dimitar} is a mixture of the regular Weyl group action and the $\rho_{\oa}$-shifted action for $\mathfrak{gl}(n)$, depending on the atypicality of the weight. More precisely, for a simple root $\alpha$ we set
$$s_{\alpha}*\lambda=\left\{
\begin{array}{ll}
s_{\alpha}\cdot\lambda=s_\alpha\lambda &\text{ if } \langle \lambda,\overline{\alpha}\rangle \not=0,\\
s_{\alpha}\circ\lambda =s_\alpha\lambda-\alpha &\text{ if } \langle \lambda ,\ol{\alpha}\rangle =0,
\end{array}
\right.$$
where $s_{\alpha}\circ\lambda:=s_{\alpha}(\lambda+\rho_{\oa})-\rho_{\oa}$. In particular, we have $s_\alpha\ast s_\alpha\ast\lambda=\lambda$ and $s_\alpha\ast s_\beta\ast\lambda =s_\beta\ast s_\alpha\ast\lambda$ if $\langle\alpha ,\beta \rangle=0$. As for basic classical Lie superalgebras we will write $xy\ast\lambda$ for $x*y*\lambda$.

The simple $\mathfrak{q}(n)$-module $L(\lambda)$ is $\alpha$-finite if and only if $\langle \alpha,\lambda\rangle\in \mZ$ and $s_\alpha\ast\lambda <\lambda$, see Statement 4.1.4 in \cite{Dimitar}.

\begin{lemma}
\label{helpqn}
For $\mu\in\fh^\ast_{\oa}$, we have $[T_\alpha L(\mu)\,: \, L(s_\alpha\ast\mu)]\not=0$ if $L(\mu)$ is $\alpha$-free.
\end{lemma}
\begin{proof}
We prove that the highest weight of $T_\alpha L(\mu)$ is $s_\alpha\ast\mu$. By using parabolic induction this becomes a $\mathfrak{q}(2)$-property. The result then follows immediately from the structure of simple $\mathfrak{q}(2)$-highest weight modules, see e.g. \cite{MR2742017}.
\end{proof}

\begin{corollary}
\label{corqn}
If $\langle \lambda,\alpha^\vee\rangle\not\in\mZ$, we have $L(s_\alpha\ast\lambda)=T_\alpha L(\lambda)$.
\end{corollary}
\begin{proof}
This follows from the combination of Lemma \ref{helpqn} and Lemma \ref{simpnonint}.
\end{proof}

For the left coset representatives $W^\lambda$ from equation \eqref{coset}, the star action is well behaved, as follows from the following lemma.

\begin{lemma}
\label{qnstarnonint}
Consider $\lambda\in\fh_{\oa}^\ast$, $w\in W^\lambda$ and two $\underline{reduced}$ expressions $s_{\alpha_1}\cdots s_{\alpha_k}$ and $s_{\beta_1}\cdots s_{\beta_k}$ for $w$. Then  we have $s_{\alpha_1}\cdots s_{\alpha_k}\ast\lambda=s_{\beta_1}\cdots s_{\beta_k}\ast\lambda$.
\end{lemma}

\begin{proof}
Mutatis mutandis Proposition \ref{cosetWeylprop}.
\end{proof}

In the following lemma we prove that this action leads to regular orbits for weakly generic weights. Close to the walls orbits of the star action can take exotic forms, containing more than one weight which is maximal with respect to the star action, see Example 4.2 (3) in \cite{Dimitar}.

\begin{proposition}
\label{orbitqnfar}
Consider $w\in W$ with two expressions $s_{\alpha_1}\cdots s_{\alpha_k}$ and $s_{\beta_1}\cdots s_{\beta_l}$. Then we have $s_{\alpha_1}\cdots s_{\alpha_k}\ast\lambda=s_{\beta_1}\cdots s_{\beta_l}\ast\lambda$ for $\lambda$ a weakly generic weight. The star orbit of a weakly generic weight is regular in the sense that it intersects each Weyl chamber exactly once.
\end{proposition}
\begin{proof}
Iteratively using Lemma \ref{noextraatyp} implies that for $p$ distinct $\{\gamma_1,\cdots\gamma_p\}\subset\Delta_{\ob}^+$ with $\langle \lambda,\overline{\gamma_i}\rangle=0$ and for $\lambda$ weakly generic, $\lambda+\gamma_1+\cdots+\gamma_p $ has exactly the same atypicalities as $\lambda$.

We will prove the following statement by induction on the length $k$. For $k$ not necessarily distinct simple elements $\{\alpha_1,\cdots,\alpha_k\}$ in $\Delta^+_{\oa}$ we have
\begin{equation}\label{eqstarq}s_{\alpha_k}\cdots s_{\alpha_1}\ast\lambda=s_{\alpha_k}\cdots s_{\alpha_1}\left(\lambda+\gamma_1+\cdots+\gamma_p\right) \qquad\mbox{with}\end{equation}
\begin{equation}\label{eq1p}\{\gamma_1,\cdots,\gamma_p\}=\{\gamma\in\Delta_{\ob}^+\,|\,\langle\overline{ \gamma},\lambda\rangle=0\mbox{ and }s_{\alpha_k}\cdots s_{\alpha_1}(\gamma)<0\}. \end{equation}

For $k=1$ this is exactly the definition of the star action on any $\lambda$. We assume equation \eqref{eqstarq} is valid for $k$ and prove it is true for $k+1$. We take $\{\gamma_1,\cdots,\gamma_p\}$ as in \eqref{eq1p} and define $\gamma_{p+1}:=s_{\alpha_1}\cdots s_{\alpha_k}(\alpha_{k+1})$, then we have
$$\{\gamma\in\Delta_{\ob}^+\,|\,\langle\overline{ \gamma},\lambda\rangle=0\mbox{ and }s_{\alpha_{k+1}}s_{\alpha_k}\cdots s_{\alpha_1}(\gamma)<0\}=$$
\begin{equation}\label{eq1p1}\begin{cases}\{\gamma_1,\cdots,\gamma_p\}\quad\mbox{if}\quad \langle\lambda,\overline{\gamma_{p+1}}\rangle\not=0\qquad\qquad\qquad\qquad\qquad\qquad\quad (1)\\
\{\gamma_1,\cdots,\gamma_p\}\cup\{\gamma_{p+1}\}\quad\mbox{if}\quad \langle\lambda,\overline{\gamma_{p+1}}\rangle=0\mbox{ and }\gamma_{p+1}>0\qquad\,\,\,\,\, (2a)\\
\{\gamma_1,\cdots,\gamma_p\}\backslash\{-\gamma_{p+1}\}\quad\mbox{if}\quad \langle\lambda,\overline{\gamma_{p+1}}\rangle=0\mbox{ and }\gamma_{p+1}<0\qquad\,\,\,\, (2b)\end{cases}
\end{equation}
We consider these three different possibilities. First we note that, as argued in the first paragraph, $s_{\alpha_k}\cdots s_{\alpha_1}(\lambda+\gamma_1+\cdots+\gamma_p) $ has exactly the same atypicalities as $s_{\alpha_k}\cdots s_{\alpha_1}(\lambda)$. This implies that we have $\langle\lambda,\overline{\gamma_{p+1}}\rangle\not=0$ if and only if $\langle s_{\alpha_k}\cdots s_{\alpha_1}\ast\lambda, \overline{\alpha_{k+1}}\rangle\not=0$.

(1) If $\langle\lambda,\overline{\gamma_{p+1}}\rangle\not=0$, the definition of the star action and induction step yield
\[s_{\alpha_{k+1}}s_{\alpha_k}\cdots s_{\alpha_1}\ast\lambda=s_{\alpha_{k+1}}s_{\alpha_k}\cdots s_{\alpha_1}\left(\lambda+\gamma_1+\cdots+\gamma_p\right) \]
with $\{\gamma_1,\cdots,\gamma_p\}$ as in equation \eqref{eq1p}, which indeed agrees with equation \eqref{eq1p1}.

(2) If $\langle\lambda,\overline{\gamma_{p+1}}\rangle=0$, the definition and induction step yield
\begin{eqnarray*}s_{\alpha_{k+1}}s_{\alpha_k}\cdots s_{\alpha_1}\ast\lambda&=&s_{\alpha_{k+1}}s_{\alpha_k}\cdots s_{\alpha_1}\left(\lambda+\gamma_1+\cdots+\gamma_p\right) -\alpha_{k+1}\\
&=&s_{\alpha_{k+1}}s_{\alpha_k}\cdots s_{\alpha_1}\left(\lambda+\gamma_1+\cdots+\gamma_p+\gamma_{p+1}\right),
\end{eqnarray*}
with $\{\gamma_1,\cdots,\gamma_p\}$ as in equation \eqref{eq1p}. This gives the correct result by equation \eqref{eq1p1} for both $(2a)$ and $(2b)$.

The obtained expression in equation \eqref{eqstarq} clearly agrees with $\widetilde{W}\tto W$. Moreover, since $\lambda+\gamma_1+\cdots+\gamma_p$ is in the same Weyl chamber as $\lambda$, $w\ast\lambda$ is in the same Weyl chamber as $w\lambda$, which implies the remainder of the lemma.
\end{proof}

We repeat Penkov's result of \cite{MR1309652} on the structure of modules with integral weakly generic highest weight.
\begin{lemma}
\label{lemPenkov}
Consider $\fg=\mathfrak{q}(n)$ and a weakly generic $\underline{\mbox{integral}}$ weight $\lambda\in\fh_{\oa}^\ast$. Define the subset of roots $S_\lambda=\{\beta\in\Delta_{\ob}^+\,|\,\langle\beta,\lambda\rangle\not=0\}$. For some $k\in\mN$, we have 
\[\Res L(\lambda)=\bigoplus_{I\subset S_\lambda}L_{\oa}(\lambda-\sum_{\beta\in I}\beta)^{\oplus k}.\]
\end{lemma}
\begin{proof}
This is a translation of the result in Theorem 2.2 in \cite{MR1309652}.
\end{proof}

Using this result we can determine the Jordan-H\"older decomposition series of $T_\alpha L(\lambda)$ entirely for $\lambda$ generic and integral.
\begin{proposition}
\label{Tgenqn}
For $\fg=\mathfrak{q}(n)$ and an $\underline{integral}$ dominant generic weight $\Lambda\in\fh_{\oa}^\ast$ we have
\[[T_\alpha L(w\ast\Lambda): L(w'\ast\Lambda)]\, \,=\,\, [T^{\oa}_\alpha L_{\oa}(w\circ\Lambda): L(w'\circ \Lambda)],\]
for arbitrary $w,w'\in W$.
\end{proposition}

\begin{proof}
We define the set of ($\fg_{\oa}$-integral dominant) weights $C_\Lambda$ as
\[\Res L( \Lambda)= \bigoplus_{\mu\in C_\Lambda}L_{\oa}(\mu)^{\oplus k} \]
with $k$ from Lemma \ref{lemPenkov}. Then for every $w\in W$ we prove the claim
\[\Res L(w\ast \Lambda)= \bigoplus_{\mu\in C_\Lambda}L_{\oa}(w\circ \mu)^{\oplus k} .\]
We prove this for $w=s_\alpha$, the full result then follows from iteratively using the same procedure. If $\Lambda$ is $\alpha$-atypical, then $s_\alpha\ast\Lambda=s_\alpha\circ\Lambda$ and by Lemma \ref{noextraatyp}
we have $S_{s_\alpha\ast\Lambda}=S_{s_\alpha\Lambda}$. Therefore
\[S_{s_\alpha\ast\Lambda}=S_{s_\alpha\Lambda}=\{\beta\in\Delta_{\ob}^+\,|\,\langle s_\alpha(\beta),\lambda\rangle\not=0\}=s_\alpha(S_\Lambda),\]
since for $\beta\in \Delta_{\ob}^+$ different from $\alpha$, we have $s_\alpha(\beta)\in\Delta_{\ob}^+$. The claim then follows immediately from Lemma \ref{lemPenkov}. Now assume that $\Lambda$ is $\alpha$-typical. Then we have $s_\alpha\ast\Lambda=s_\alpha\Lambda$ and similarly as above
\[S_{s_\alpha\ast\Lambda}=s_\alpha (S_\Lambda)\backslash \{-\alpha\}\,\cup \{\alpha\}.\]
Applying Lemma \ref{lemPenkov} then also yields the claim.

One consequence of this claim is that all the modules $L(w\ast\Lambda)$ are generic. The proposition then follows from Lemmata \ref{twistinggeneric} and \ref{resindT}.
\end{proof}

\section{Primitive ideals for $\mathfrak{q}(n)$}
\label{secqn}

For typical weights of $\mathfrak{q}(n)$, the star action equals the usual unshifted action (since $\rho=0$). In the following lemma we derive all inclusions of primitive ideals corresponding to regular strongly typical central characters.
\begin{lemma}
Consider $\fg=\mathfrak{q}(n)$ and a strongly typical regular dominant weight $\Lambda$, then  we have $J(w\Lambda)\subseteq J(w'\Lambda)$ for $w,w'\in W$ if and only if $I(w\circ\Lambda)\subseteq I(w'\circ\Lambda)$.
\end{lemma}
\begin{proof}
Proposition 2 in \cite{Frisk} states that for $k=2^{\lfloor (n-1)/2\rfloor}$ we have
\[\left(\Ind L_{\oa}(w\circ\widetilde{\Lambda})\right)_{\chi}=k\, L(w\Lambda)\quad\mbox{and}\quad \left(\Res L(w\Lambda)\right)_{\widetilde{\chi}}=k\, L_{\oa}(w\circ \widetilde{\Lambda}).\]
The statements then follow from Lemma \ref{IplusAnn} applied to $\fg_{\oa}$ and Lemma \ref{induct}.
\end{proof}

According to Proposition \ref{orbitqnfar}, we can unambiguously define $w\ast\lambda$ for $w\in W$ and a weakly generic weight $\lambda$. The main result in this section is the full classification of primitive ideals in the generic domain, which is given in the following theorem.

\begin{theorem}
\label{mainqn}
Consider a weakly generic dominant weight $\Lambda\in\fh_{\oa}^\ast$ with $w_0\ast\Lambda$ generic and an arbitrary $\mu\in\fh^\ast_{\oa}$. We have
\[J(w_1\ast\Lambda)= J(\mu)\quad\Leftrightarrow\quad \exists \,w_2\in W\, \mbox{ with }\, \mu=w_2\ast\Lambda\,\mbox{ and }\\ I(w_1\circ \Lambda)= I(w_2\circ\Lambda).\]
Furthermore, if for $x_1,x_2\in W^\lambda$ and $u_1,u_2\in W_\lambda$ there are $v_1,v_2\in W_\lambda$ such that $v_1$ is in the same left cell as $u_1$ and $v_2$ is in the same left cell as $u_2$ and $v_2\le_L v_1$ in the left Bruhat order then we have $J(x_1u_1\ast\Lambda)\subseteq J(x_2u_2\ast\Lambda)$.
\end{theorem}

\begin{remark}{\rm
The condition on $v_1,v_2\in W_\lambda$ in the theorem above implies that that $v_2\le v_1$ in left Kazhdan-Lusztig order (see Definition 14.15 in \cite{MR0721170}), which is equivalent to the condition $I(v_1\circ \Lambda)\subseteq I(v_2\circ \Lambda)$
(see Subsections~14.15 and 16.4 in \cite{MR0721170} and also \cite{MR0575938}). For $n\le 5$, these two conditions are equivalent, see \cite{MR3003677}. According to Lemma \ref{farneq} and Lemma \ref{translclass}, we have therefore classified all inclusions between generic primitive ideals for $\mathfrak{q}(n)$ for $n\le 5$. In the subsequent Proposition \ref{qnint} we classify all inclusions between integral generic weights for $\mathfrak{q}(n)$ for arbitrary $n$.}
\end{remark}

\begin{lemma}
\label{helpqncor}
Consider $\fg=\mathfrak{q}(n)$. For $\alpha$ a simple root and $\lambda\in\fh^\ast_{\oa}$ we have the inclusions
\begin{itemize}
\item $J(s_\alpha\ast\lambda)=J(\lambda)$ if $\langle \lambda,\alpha^\vee\rangle\not\in\mZ$
\item $J(s_\alpha\ast\lambda)\subset J(\lambda)$ if $\langle \lambda,\alpha^\vee\rangle\in\mZ$ and $s_\alpha\ast\lambda < \lambda$.
\end{itemize}
\end{lemma}

\begin{proof}
Similarly to the proof of Lemma \ref{primsl2}, this can be reduced to the corresponding result for $\mathfrak{q}(2)$ in Proposition \ref{Primq2}.
\end{proof}

We need the following lemma on the left coset representatives $W^\lambda$ of the integral Weyl group $W_\lambda\subset W$.

\begin{lemma}
\label{eqprimqnhelp}
For an arbitrary weight $\lambda\in\fh^\ast_{\oa}$ and $w\in W^\lambda$ with a $\underline{reduced}$ expression $s_{\alpha_1}\cdots s_{\alpha_k}$ we have $J(\lambda)=J(s_{\alpha_1}\cdots s_{\alpha_k}\ast\lambda)$.
\end{lemma}
\begin{proof}
This is a consequence of Lemma \ref{helpqncor} as in the proof of Proposition \ref{cosetWeylprop}.
\end{proof}

\begin{proof}[Proof of Theorem \ref{mainqn}]
First, we focus on proving
\[J(w_1\ast\Lambda)= J(w_2\ast\Lambda)\Leftrightarrow I(w_1\circ \Lambda)= I(w_2\circ\Lambda).\]
The ``if'' part follows from Lemma \ref{farneq} and Lemma \ref{translclass}, so we prove the ``only if'' part. First we note that all the modules $L(w\ast\Lambda)$ are generic. This follows from the fact that $L(w_0\ast\Lambda)$ is generic (see Corollary \ref{corgen}) and the combination of Lemma \ref{helpqn} and Lemma \ref{twistinggeneric}.

We introduce the notation $\Delta_{\oa}^+(\mu)$ for the positive roots in $\Delta_{\oa}(\mu)$ defined in equation \eqref{deltalambda}. All equalities $ I(w_1\circ \Lambda)= I(w_2\circ\Lambda)$ can be derived from those of the form
\begin{itemize}
\item $I(w'w\circ\Lambda)= I(w\circ\Lambda)$ for $w\in W$ and $w'\in W^{w\Lambda}$;
\item $I(s_\beta s_\alpha w\circ \Lambda)=I(s_\alpha w\circ\Lambda)$ for $w\in W$, $\alpha,\beta$ simple in $\Delta_{\oa}^+(w\Lambda)$ with $\langle \alpha,\beta^\vee\rangle=-1$, $w^{-1}(\alpha)\in\Delta_{\oa}^+$ and  $w^{-1}(\beta)\in\Delta_{\oa}^+$;
\end{itemize}
see Sections 5.25 and 5.26 in \cite{MR0721170}. Lemma \ref{eqprimqnhelp} implies that for the first type of classical equalities we have $J(w'w\ast\Lambda)= J(w\ast\Lambda)$. So we focus on the second type of equality. 

First we assume that $\alpha$ and $\beta$ are actually simple in $\Delta_{\oa}^+\supset \Delta_{\oa}^+(w\Lambda)$. Theorems 6.3 and 7.8 in \cite{MR2032059}, then imply that 
\begin{equation}\label{eqnew}[T^{\oa}_\beta L_{\oa}(s_\beta s_\alpha w\circ \Lambda): L_{\oa}(s_\alpha w\circ \Lambda)]=1 \quad\mbox{and}\quad [T^{\oa}_\alpha L_{\oa}(s_\alpha w\circ \Lambda): L_{\oa}(s_\beta s_\alpha w\circ \Lambda)]=1 ,\end{equation}
since $\dim\Ext^1_{\cO_{\oa}}(L_{\oa}(s_\beta s_\alpha w\circ \Lambda), L_{\oa}(s_\alpha w\circ \Lambda))=1$.

Lemmata \ref{twistinggeneric} and \ref{resindT} therefore imply that there are weights $\kappa_1,\cdots, \kappa_j$ (with $j\ge 1$) inside the chamber of $s_\alpha w\Lambda$, such that the generic modules $\{L(\kappa_i),i=1,\cdots,j\}$ are the ones corresponding to that Weyl chamber appearing in the Jordan-H\"older decomposition of $T_\beta L(s_\beta s_\alpha w\ast\Lambda)$. Lemma \ref{helpqn} states that we can choose $\kappa_1= s_\alpha\ast\Lambda$. For each of these weights $\kappa_i$, there are similarly weights $\kappa_{i,1},\cdots \kappa_{i, j_i}$ (with $j_i\ge 1$) inside the Weyl chamber of $s_\beta s_\alpha w\Lambda$, such that the generic modules $\{L(\kappa_{i l}),l=1,\cdots, j_i\}$ are the ones corresponding to that Weyl chamber appearing in the Jordan-H\"older decomposition of $T_{\alpha} L(\kappa_i)$. Equation \eqref{eqnew} implies that
\[\ch L(s_\beta s_\alpha w\ast\Lambda)=\bigoplus_{i=1}^j\bigoplus_{l=1}^{j_i}\ch L(\kappa_{il})\]
which shows that $j=1$, $j_1=1$. We conclude that $[T_\alpha L(s_\alpha w\ast\Lambda): L(s_\beta s_\alpha w\ast\Lambda)]\not=0$ and $[T_{\beta} L(s_\beta s_\alpha w\ast\Lambda):L(s_\alpha w\ast\Lambda)]\not=0$, so Lemma \ref{inclannT} implies the equality $J(s_\beta s_\alpha w\ast\Lambda)=J(s_\alpha w\ast\Lambda)$.

Now we consider the second type of equality with general $\alpha,\beta$ simple in $\Delta_{\oa}^+(w\Lambda)$. According to Lemma 5.17(b) in \cite{MR0721170} or Lemma 15.3.24(b) in \cite{MR2906817}, there exist a $u\in W^{w\Lambda}$ for which $\alpha'=u(\alpha)$ and $\beta'=u(\beta)$ are simple in $\Delta_{\oa}^+$ . This also implies $s_{\alpha}=u^{-1}s_{\alpha'} u$ and $s_{\beta}=u^{-1}s_{\beta'} u$. Lemma \ref{eqprimqnhelp} implies
\[J(s_\beta s_\alpha w\ast \Lambda)= J(s_{\beta'} s_{\alpha'} uw\ast \Lambda)\quad\mbox{and }\quad J(s_\alpha w\ast\Lambda)=J(s_{\alpha'} u w\ast\Lambda).\]
Since $(uw)^{-1}(\alpha')=w^{-1}(\alpha)\in \Delta^+_{\oa}$ (and likewise for $\beta$), the equality $J(s_\beta s_\alpha w\ast \Lambda)= J(s_\alpha w\ast\Lambda)$ follows from the previous case.

Since we already established the equalities between weights in the same star orbit, Lemma \ref{farneq2} implies that there can be no equalities of the form $J(w\ast\Lambda)=J(\mu)$ for $\mu$ weakly generic but not in the star orbit of $\Lambda$. If there would be an equality $J(w\ast\Lambda)=J(\mu)$ for a non-weakly generic weight $\mu$, then no module corresponding to $\chi^{\oa}_\mu$ can appear in the module $\Res L(w\ast\Lambda)$ since $L(w\ast\Lambda)$ is generic. Corollary~\ref{corIAnn} for $\fg_{\oa}$ then implies this equality can not appear.

The property on the inclusions follows from Lemmata \ref{helpqncor} and \ref{eqprimqnhelp}.  
\end{proof}

Now we show that for integral generic modules the inclusions can be classified by using the consequences of Penkov's result on the structure of generic modules. 

\begin{proposition}
\label{qnint}
Consider $\fg=\mathfrak{q}(n)$, a generic integral dominant $\Lambda\in\fh^\ast$ and $w_1,w_2\in W$. Then we have
\[J(w_1\ast\Lambda)\subseteq J(w_2\ast\Lambda)\,\,\Leftrightarrow\,\, I(w_1\circ\Lambda)\subseteq J(w_2\circ\Lambda).\]
\end{proposition}

\begin{proof}
According to Subsections~14.15 and 16.4 in \cite{MR0721170}, $I(w_1\circ\Lambda)\subseteq I(w_2\circ\Lambda)$ if and only if $w_1$ is less than or equal to $w_2$ with respect to the left Kazhdan-Lusztig order. We claim that the latter
condition is equivalent to existence of a sequence $\{\alpha_1,\cdots,\alpha_p\}$ of simple roots in $\Delta_{\oa}^+$ and Weyl group elements $\{w^{(0)},w^{(1)},\cdots, w^{(p)}\}$ with $w^{(0)}=w_1$ and $w^{(p)}=w_2$ such that 
\begin{equation}\label{eqeq762}
[T^{\oa}_{\alpha_i}L_{\oa}(w_{i-1}\circ\Lambda): L_{\oa}(w_i\circ\Lambda)]\not=0\,\,\text{ for }\,\,i=1,\cdots,p. 
\end{equation}
Indeed, by e.g. Appendix in \cite{MO}, the combinatorics of the action of twisting functors (for
$\mathfrak{g}_{\oa}$) is left-right dual to the combinatorics of the action of the usual projective functors. 
Hence the above statement is equivalent to a similar statement for projective functors in which the left
Kazhdan-Lusztig order is changed to the right Kazhdan-Lusztig order. In the latter situation the statement is
standard, see for example Lemma~13 in \cite{MM}.

Proposition \ref{Tgenqn} states that \eqref{eqeq762} is equivalent to the property 
\begin{displaymath}
[T_{\alpha_i}L(w_{i-1}\ast\Lambda): L(w_i\ast\Lambda)]\not=0\,\,\text{ for }\,\,i=1,\cdots,p.
\end{displaymath}
The result therefore follows from Lemma~\ref{inclannT}.
\end{proof}


\section{Primitive ideals for $\mathfrak{osp}(m|2n)$}
\label{primosp}

In this section we apply the ideas of the previous sections to study the primitive ideals of $\fg=\mathfrak{osp}(m|2n)$, with $m>2$. Primitive ideals for $\mathfrak{osp}(m|2n)$ have not been studied previously, except for the special cases $\mathfrak{osp}(1|2n)$ in \cite{MR1479886}, $\mathfrak{osp}(2|2n)$ in \cite{MR1362685} and a study of analogues of the Joseph ideal in \cite{CSS}.

We obtain a classification of primitive ideals and their inclusions in the generic region. This result states that these inclusions are in correspondence with the ones for the underlying Lie algebra where the action of the Weyl group is given through the star action. As proved in Theorem \ref{Sstarfar}, it is not important which star action is chosen in this generic region. However, in contrast to basic classical Lie superalgebras of type I, the star action in the generic region is not identical to the usual $\rho$-shifted action of the Weyl group, see Lemma \ref{starexexp}. The newly introduced notion of the star action is therefore, as for $\mathfrak{q}(n)$, an essential concept for the study of primitive ideals in the generic region. 

Although all star actions become equivalent in the generic region, it will be more transparent to explicitly introduce a star action that is different from the one from Subsection \ref{exstarosp} for our purposes. The reference Borel subalgebra remains the standard Borel subalgebra $\fb$ with system of positive roots as in the beginning of Subsection \ref{exstarosp}. The second Borel subalgebra we will use is $\widetilde{\fb}$, with simple positive roots given by
\begin{eqnarray*}
 \epsilon_1-\epsilon_2,\cdots, \epsilon_{d-1}-\epsilon_d,\epsilon_d-\delta_1,\delta_1-\delta_2,\cdots,\delta_{n-1}-\delta_n, \delta_n\qquad \mbox{for }m=2d+1,\\
\epsilon_1-\epsilon_2,\cdots, \epsilon_{d-1}-\epsilon_d,\epsilon_1-\delta_1\delta_1-\delta_2,\cdots,\delta_{n-1}-\delta_n, 2\delta_n\qquad \mbox{for }m=2d.
\end{eqnarray*}
Then we define a star action map $\cB$ by $$\cB(\alpha)=\begin{cases}\fb&\mbox{ if $\alpha$ is a simple root of }\mathfrak{so}(m)\\
\widetilde{\fb}&\mbox{ if $\alpha$ is a simple root of }\mathfrak{sp}(2n).\end{cases}$$

The corresponding star action $\ast_{\cB}$ clearly has the property that it leads to an action of $W(\mathfrak{so}(m))$ and $W(\mathfrak{sp}(2n))$. Because of this interesting property we refer to $\ast_\cB$ as {\it the} star action for $\mathfrak{osp}(m|2n)$ and thus denote it by $\ast$. The reason why it does not correspond to an action of the full Weyl group is that in general $w_1w_2\ast\mu\not= w_2w_1\ast\mu$ for $w_1\in W(\mathfrak{so}(m))$, $w_2\in W(\mathfrak{sp}(2n))$ and $\mu\in\fh^\ast$. As proved in Theorem \ref{Sstarfar}, in the generic region the star action $\ast$ becomes equal to $\ast'$ from Subsection \ref{exstarosp} and furthermore $w_1w_2\ast\mu= w_2w_1\ast\mu$ for $\mu$ weakly generic.

\begin{theorem}
\label{mainosp}
Consider $\fg=\mathfrak{osp}(m|2n)$ and a dominant weakly generic $\Lambda\in\fh^\ast$ with $w_0\ast\Lambda$ generic. If we have the equality $J(w\ast\Lambda)=J(\mu)$ for some $\mu\in\fh_{\oa}^\ast$, then $\mu$ is in the star orbit of $\Lambda$. Furthermore, we have the equivalence
\[J(w_1\ast\Lambda)\subseteq J(w_2\ast\Lambda)\quad\Leftrightarrow\quad I(w_1\circ \Lambda)\subseteq I(w_2\circ\Lambda).\]
\end{theorem}

\begin{remark}{\rm
If only the displayed equivalence is needed, a sufficient condition is that every weight in the star orbit of $\Lambda$ is weakly generic, which is satisfied if an arbitrary element in the orbit of $\Lambda$ is generic.}
\end{remark}

Before proving this, we make a useful remark concerning the primitive ideals for $\fg_{\oa}$. For $\fg=\mathfrak{osp}(m|2n)$ we have $\fg_{\oa}=\mathfrak{so}(m)\oplus\mathfrak{sp}(2n)$. Thus we can naturally decompose weights $\mu\in\fh^\ast$ as $\mu=\mu'+\mu''$ with $\mu'$ a weights for $\mathfrak{so}(m)$ and $\mu''$ a weight for $\mathfrak{sp}(2n)$. We also denote the corresponding primitive ideals for $\mathfrak{so}(m)$ and $\mathfrak{sp}(2n)$ by $I'(\mu')$ and $I''(\mu'')$, respectively.

\begin{lemma}
\label{helposp}
With notation as above and $\mu,\nu\in\fh^\ast$, we have
\[I(\mu)\subseteq I(\nu)\quad\Leftrightarrow\quad I'(\mu')\subseteq I'(\nu')\mbox{ and }I''(\mu'')\subseteq I''(\nu'').\]
\end{lemma}
\begin{proof}
As $U(\fg_{\oa})\cong U(\mathfrak{so}(m))\otimes U(\mathfrak{sp}(2n))$, it follows easily that 
$$ I(\mu) =I'(\mu')\otimes U(\mathfrak{sp}(2n))+ U(\mathfrak{so}(m))\otimes I''(\mu'')$$
and
$$ I'(\mu')=I(\mu)\cap (U(\mathfrak{so}(m))\otimes 1),\qquad I''(\mu'')=I(\mu)\cap (1\otimes U(\mathfrak{sp}(2n))).$$
The claims then follow immediately from these equalities.
\end{proof}

\begin{proof}[Proof of Theorem \ref{mainosp}]

First we prove the second statement. The ``if'' part follows from Lemma \ref{farneq} and Lemma \ref{translclass}. For the ``only if'' part we decompose the elements of the Weyl group as $w_1=w_1'w_1''$ and $w_2=w_2'w_2''$ with $w_1',w_2'\in W(\mathfrak{sp}(2n))$ and $w_1'',w_2''\in W(\mathfrak{so}(m))$. According to Lemma \ref{helposp}, the inclusions of primitive ideals can then be refined to $$ I(w_1\circ\Lambda)\subseteq I(w_1'w_2''\circ\Lambda)\subseteq I(w_2\circ \Lambda). $$
Therefore it suffices to consider cases where $w_2w_1^{-1}$ is either in the Weyl group of $\mathfrak{so}(m)$ or $\mathfrak{sp}(2n)$. The first case follows immediately from Corollary \ref{corIan} since by definition $w\ast=w\circ=w\cdot$ for $w\in W(\mathfrak{so}(m))$. For the second case we denote $w'':=w_1''=w_2''$. According to Lemma \ref{translclass}, we can replace $\Lambda$ by $(w'')^{-1}\circ \widetilde{w''\circ\Lambda}$ in the inclusions of annihilator ideals of $\fg_{\oa}$ to get
$$I(w_1'\circ  \widetilde{w''\circ\Lambda})\subseteq I(w_2'\circ  \widetilde{w''\circ\Lambda} ).$$ Corollary \ref{corIan} then yields
$$\Ann_{U(\fg)}L^{(\widetilde{\fb})}(w_1'\circ  \widetilde{w''\circ\Lambda})\subseteq \Ann_{U(\fg)}L^{(\widetilde{\fb})}(w_2'\circ  \widetilde{w''\circ\Lambda} ),$$
which is equivalent to
$$J(w_1'w''\ast\Lambda)=J(w_1'\ast w''\circ\Lambda)\subseteq J(w_2'\ast w''\circ\Lambda)=J(w_2'w''\ast\Lambda).$$
The first statement follows from the second and Lemma \ref{farneq2} as in the proof of Theorem \ref{mainqn}.
\end{proof}


\section{Primitive ideals for classical Lie superalgebras of type I}
\label{secI}

We formulate the results in this section in terms of the distinguished system of positive roots, i.e. the reference Borel subalgebra satisfies $\fb\subset \fg_0\oplus\fg_1$.

\subsection{General statements}

\begin{theorem}
\label{mainthmI}
Consider $\fg$ a classical Lie superalgebra of type I with distinguished system of positive roots. For $\lambda,\mu\in \fh^\ast$, we have
\begin{equation}
\label{thmI1}\quad J(\mu)=J(\lambda)\quad \Leftrightarrow  \quad I(\mu)=I(\lambda).\end{equation}
Furthermore, we have
\begin{equation}\label{thmI3}\quad J(w'\cdot\lambda)\subset J(w\cdot\lambda)\quad \Leftrightarrow  \quad I(w'\cdot\lambda)\subset I(w\cdot\lambda)\end{equation}
for $w,w'\in W$.
\end{theorem}

The first part, which determines the fibres of the surjective mapping of annihilator ideals of simple highest weight modules onto the set of primitive ideals as introduced by Musson in \cite{MR1149625}, recovers the main result of \cite{MR1362685} obtained by Letzter. The remainder of this subsection is mainly devoted to proving Theorem \ref{mainthmI}. The following lemma implies that equalities of primitive ideals can only occur between modules with highest weights in the same orbit. Contrary to the classical case, this does not follow from the consideration of central characters.
\begin{lemma}
\label{typeIorbitprop}
If for $\lambda,\mu\in\fh^\ast$ we have $J(\mu)=J(\lambda)$, then $\mu=w\cdot\lambda$ for some $w\in W$.
\end{lemma}
\begin{proof}
Assume that $\mu$ and $\lambda$ belong to different orbits. Denote by $\hat{\mu}$ and $\hat{\lambda}$ the $W$-maximal elements in their orbits. Without loss of generality we take $\hat{\lambda}\not<\hat{\mu}$. There is always an integral dominant weight $\nu$ such that $\Lambda=\hat{\lambda}+\nu$ is typical. This follows from the property
\[T\cap (\hat{\lambda}+\cP^+)\not=\emptyset,\]
with $T\subset \fh^\ast$ the open subset of typical weights and $\hat{\lambda}+\cP^+$ a dense set in the Zariski topology.

 Based on the standard filtration of the tensor product of a Verma module with a finite dimensional module, one obtains that
\[\left(M(w\cdot\hat{\mu})\otimes L(\nu)\right)_{\chi_{\Lambda}}=0\]
for any $w\in W$ and thus $\left(L({\mu})\otimes L(\nu)\right)_{\chi_{\Lambda}}=0$. On the other hand, based on the property of highest weights it follows that $\left(L(\hat{\lambda})\otimes L(\nu)\right)_{\chi_\Lambda}\not=0$. Lemma \ref{starSTHW} (for the special case of the distinguished system of positive roots and $\cB$ trivial), and Lemma \ref{tensorfd} imply
\[\left(L( w\cdot \hat{\lambda})\otimes L(\nu)\right)_{\chi_\Lambda}\quad \mbox{is a subquotient of}\quad T_{\alpha}\left(L(s_\alpha w\cdot\hat{\lambda})\otimes L(\nu)\right)_{\chi_\Lambda}\,\mbox{ if }\, s_\alpha w>w.\]
The leads to the fact that $\left(L(w\cdot\hat{\lambda})\otimes L(\nu)\right)_{\chi_\Lambda}\not=0$ for any $w\in W$.
Thus, in particular, we obtain that $\left(L({\lambda})\otimes L(\nu)\right)_{\chi_{\Lambda}}\not=0$. Corollary \ref{corIAnn} therefore implies that $\hat{\mu}=\hat{\lambda}$.
\end{proof}

\begin{lemma}
\label{inclorbitI}
We have $J(w'\cdot\lambda)\subset J(w\cdot\lambda)\,\, \Rightarrow  \,\, I(w'\cdot\lambda)\subset I(w\cdot\lambda)$.
\end{lemma}

\begin{proof}
Since it is a statement concerning orbits of the Weyl group, we can assume that $\lambda$ is $W$-maximal without loss of generality. As in the proof of Lemma \ref{typeIorbitprop} we take a $\nu\in\cP^+$ such that $\Lambda=\lambda+\nu$ is typical.

If $\lambda$ is regular (and therefore $\Lambda$ as well), then $(M(w\cdot\lambda)\otimes L(\nu))_{\chi_\Lambda}=M(w\cdot\Lambda)$, with the analogous property also holding for the dual Verma module. This immediately implies that $(L(w\cdot\lambda)\otimes L(\nu))_{\chi_\Lambda}$ is either $L(w\cdot\Lambda)$ or trivial. As in the proof of Lemma \ref{typeIorbitprop}, it is not trivial. Lemmata \ref{Anntens} and \ref{IplusAnn} therefore imply $J(w'\cdot\lambda)\subset J(w\cdot\lambda) \Rightarrow J(w'\cdot\Lambda)\subset J(w\cdot\Lambda) $. The result then follows from Theorem \ref{typeItyp}.

In general $\lambda$ can be singular and $\Lambda$ might have less singularities than $\lambda$. Therefore the modules $(L(w\cdot\lambda)\otimes L(\nu))_{\chi_\Lambda}$ might be non-simple. They are still non-zero by the proof of Lemma \ref{typeIorbitprop} and we denote $N(w):=(L(w\cdot\lambda)\otimes L(\nu))_{\chi_\Lambda}$ for any $w\in W$. Then by the same reason as above we have $$J(w'\cdot\lambda)\subset J(w\cdot\lambda) \Rightarrow \Ann_{U(\fg)}N(w')\subset \Ann_{U(\fg)}N(w).$$
Lemma \ref{IplusAnn} applied to $\fg_0$ then implies
$$J(w'\cdot\lambda)\subset J(w\cdot\lambda) \Rightarrow \Ann_{U(\fg_0)} (\Res N(w'))_{\chi_\Lambda^0}\subset \Ann_{U(\fg_0)} (\Res N(w))_{\chi^0_\Lambda}.$$

As already noted in the proofs of Theorems \ref{typeItyp} and \ref{typeIItyp}, we have $(\Res L(w\cdot\Lambda))_{\chi^0_\Lambda}=L_0(w\cdot\Lambda)$ for any $w\in W$. This implies that, applying the translation functor from $(\cO_{\oa})_{\chi^0_\Lambda}$ to $(\cO_{\oa})_{\chi^0_\lambda}$, see chapter 7 in \cite{MR2428237}, takes $\Res N(w)$ to $L(w\cdot\lambda)^{\oplus k}$ for some $k\in\mN$. According to Lemma 5.4 in \cite{MR0721170}, these translation functors induce morphisms of the posets of annihilator ideals of modules belonging to the respective blocks. This also follows from Lemmata \ref{Anntens} and \ref{IplusAnn}. This concludes the proof.
\end{proof}

\begin{remark}{\rm
The proof above for singular weights can be simplified if $\fg$ is $\mathfrak{sl}(m|n)$ with $m\not=n$ or $\mathfrak{osp}(2|2n)$. Then $\nu$ can be chosen as a multiple of $\rho_1$. Since this is orthogonal to all even roots, singular orbits can always be translated to typical orbits with the same singularities.}
\end{remark}

\begin{proof}[Proof of Theorem \ref{mainthmI}.]
The ``if'' part of equation \eqref{thmI1} is given in Lemma \ref{typeIorbitprop}. The ``only if'' part in equations \eqref{thmI1} and \eqref{thmI3} is exactly Lemma \ref{corIan} for $\fl=\fg_0$. The ``if'' part in equation \eqref{thmI3} is Lemma \ref{inclorbitI}.
\end{proof}

For the remainder of this section we consider $\fg=\mathfrak{sl}(m|n)$ with $m\not=n$ or $\fg=\mathfrak{osp}(2|2n)$. The even subalgebras of $\mathfrak{sl}(m|n)$ and $\mathfrak{osp}(2|2n)$ (respectively $\mC\oplus\mathfrak{sl}(m)\oplus\mathfrak{sl}(n)$ and $\mC\oplus\mathfrak{sp}(2n)$) have a one dimensional centre. We choose the element $H\in \mathfrak{z}(\fg_{\oa})\subset\fh$ of this centre which is normalised by the relation $\alpha(H)=1$ for all $\alpha\in \Delta_1^+$. This implies, in particular, that $[H,Y]=-Y$ for $Y\in\fg_{-1}$.

\begin{definition}
\label{defdlambda}
For any weight $\lambda\in\fh^\ast$ we define $d_\lambda\in \mZ_+$ (with $0\le d_\lambda\le \dim \fg_1$) as
\[d_\lambda=\max\{k\in\mZ_+ \mbox{ such that there are }\alpha_1,\cdots,\alpha_k\in\Delta_1^+\mbox{ for which }Y_{\alpha_1}\cdots Y_{\alpha_k}v^+_{\lambda}\not=0\},\]
where $v^+_\lambda$ denotes a highest weight vector of $L(\lambda)$.
\end{definition}

\begin{lemma}
\label{helpsing}
For $\lambda,\mu\in\fh_{\oa}^\ast$, the inclusion $J(\lambda)\subseteq J(\mu)$
implies existence of $p\in \mZ_+$ such that
\begin{itemize}
\item $\mu(H)=\lambda(H)-p$
\item $d_\mu \le d_\lambda -p$.
\end{itemize}
\end{lemma}
\begin{proof}
The condition $\mC[H]\cap \Ann_{ U(\fg)}L(\lambda)\subseteq \mC[H]\cap\Ann_{ U(\fg)}L(\mu)$ implies that the polynomial
\[(H-\lambda(H))(H-\lambda(H)+1)\cdots (H-\lambda(H)+d_\lambda)\]
must be divisible by
\[(H-\mu(H))(H-\mu(H)+1)\cdots (H-\mu(H)+d_\mu).\]
Thus we have $\lambda(H)\ge\mu(H)$ and $\lambda(H)-d_\lambda \le \mu(H)-d_\mu$.
\end{proof}

\subsection{Singly atypical characters for $\mathfrak{sl}(m|n)$ and $\mathfrak{osp}(2|2n)$}

Singly atypical weights $\lambda$ have one atypical root $\gamma$, $\langle \lambda+\rho,\gamma\rangle=0$, such that all other atypical roots $\gamma'$ satisfy $\langle \gamma,\gamma'\rangle\not=0$. Lemma \ref{regatyp1} implies that regular singly atypical weights therefore have exactly one atypical root. In this subsection we derive results on inclusions between primitive ideals for singular atypical characters for $\mathfrak{sl}(m|n)$ with $m\not=n$ and $\mathfrak{osp}(2|2n)$, based on the treatment of these characters by Van der Jeugt et al in \cite{MR1092559, MR1063989}.

\begin{theorem}
For integral dominant singly atypical weights $\Lambda_1,\Lambda_2\in\cP^+$ and elements of the Weyl group $w_1,w_2\in W$, the inclusion 
\[\Ann_{ U(\fg)}L(w_1\cdot\Lambda_1)\subseteq\Ann_{ U(\fg)}L(w_2\cdot \Lambda_2)\]
implies $\Lambda_1=\Lambda_2$.
\end{theorem}
\begin{proof}
According to the Harish-Chandra isomorphism, see Section 13.1 in \cite{MR2906817}, the condition $\chi_{\Lambda_1}=\chi_{\Lambda_2}$ implies that $\Lambda_1$ must be inside the $\rho$-shifted Weyl group orbit of $\Lambda_2+k\gamma$, for $\gamma$ the atypical root of $\Lambda_2$, .

Equation \eqref{thmI3} in Theorem \ref{mainthmI} implies $J(w_2\cdot\Lambda_2)\subseteq J(\Lambda_2)$, so it suffices to prove the statement for $w_2=1$. Now we can make use of the procedure in \cite{MR1092559, MR1063989} used to obtain the character formulae for singly atypical modules. We write $d:=\dim\fg_{-1}$.

Theorem 4.3 in \cite{MR1063989} implies that for every singly atypical integral dominant weight $\Lambda$, there is one singly atypical integral dominant weight $\Lambda'$ such that we have the short exact sequence
\[L(\Lambda')\hookrightarrow K(\Lambda)\tto L(\Lambda),\]
with $K(\Lambda)=\ind^{\fg}_{\fg_0+\fg_1}L_0(\Lambda)$ being the Kac module. Since $L(\Lambda')$ contains $\Lambda^{d} \fg_{-1}\,L_0(\Lambda)$, we get the equality 
\begin{equation}\label{eqsing}\Lambda(H)-d=\Lambda'(H)-d_{\Lambda'},\end{equation}
for $d_{\Lambda'}$ as in Definition \ref{defdlambda}.

The procedure in Section 6 of \cite{MR1063989} then reveals that starting from $\Lambda'$ with atypical root $\gamma$, the weight $\Lambda$ is given as the dominant weight in the orbit of $\Lambda'+k\gamma$, for $k\in\mN$ the smallest $k$ such that $\Lambda'+k\gamma$ is regular. Equation \eqref{eqsing} implies this minimal $k$ is equal to $d-d_{\Lambda'}$. 

Now we make the identification $\Lambda_2=\Lambda'$. Lemma \ref{helpsing} and the property $\chi_{\Lambda_1}=\chi_{\Lambda_2}$ imply that $\Lambda_1$ is in the orbit of $\Lambda_2+p\gamma$ for $0<p\le d_{\Lambda_1}-d_{\Lambda_2}$. On the other hand, the procedure above implies $p\ge d-d_{\Lambda_2}$. This is only possible if $d_{\Lambda_1}=d$, but that would imply that $\Lambda_1$ is typical.
\end{proof}

The next lemma proves that there are extra inclusions between regular and singular singly atypical highest weights.

\begin{lemma}
Consider a regular singly atypical weight $\lambda$, with positive atypical root $\gamma$ such that we have $\langle \lambda+\rho,\alpha^\vee\rangle \in\mN$ for $\alpha$ a simple even root. If $\langle \lambda+\gamma+\rho,\alpha\rangle=0$, then $J(\lambda+\gamma)\subset J(\lambda)$.
\end{lemma}

\begin{proof}
We set $\gamma=\epsilon_j-\delta_i$. We choose the case $\alpha=\epsilon_{j-1}-\epsilon_j$ and $\langle \lambda+\gamma+\rho,\alpha\rangle=0$, the other case being similar. We consider a Borel subalgebra $\hat{\fb}$ with simple roots given by
\[\epsilon_1-\epsilon_2,\cdots,\epsilon_{j-3}-\epsilon_{j-2},\epsilon_{j-2}-\delta_1,\delta_1-\delta_2,\cdots,\delta_{i-1}-\delta_i,\delta_i-\epsilon_{j-1},\]\[\epsilon_{j-1}-\epsilon_j,\cdots, \epsilon_{m-1}-\epsilon_m,\epsilon_m-\delta_{i+1},\delta_{i+1}-\delta_{i+2},\cdots,\delta_{n-1}-\delta_n.\]
By carrying out the odd reflections, we obtain $s_\alpha\ast_{\hat{\fb}}\lambda=\lambda+\gamma$ as in Example \ref{exMusson}. The conclusion then follows from Corollary \ref{corastS}.
\end{proof}


\section{Annihilator ideals of Verma modules }
\label{secVerma}

The result of Duflo in \cite{MR0399194} states that \[\Ann_{ U(\fg_{\oa})}M_{\oa}(w\cdot\lambda)= U(\fg_{\oa})\mathfrak{m}^{\oa}_{\chi^{\oa}_\lambda},\] where $\chi^{\oa}_\lambda:\mathcal{Z}(\fg_{\oa})\to\mC$ is the corresponding central character and $\mathfrak{m}^{\oa}_{\chi^{\oa}_\lambda}\in\Spec \mathcal{Z}(\fg_{\oa})$ is defined as $\ker\chi^{\oa}_{\lambda}$. This result was extended to strongly typical central characters for basic classical Lie superalgebras by Gorelik in \cite{MR1862799}. Below we prove a weaker statement for $\mathfrak{q}(n)$.

\begin{proposition}
Consider $\fg=\mathfrak{q}(n)$ and $\lambda$ a strongly typical weight. Then the annihilator ideal of all Verma modules in $\cO_{\chi_\lambda}$ are identical,
\[\Ann_{ U(\fg)}M(\lambda)=\Ann_{ U(\fg)}M(w\cdot\lambda)\]
for any $w\in W$ with $w\cdot\lambda=w\lambda$.
\end{proposition}

\begin{proof}
First we prove the claim for $\Lambda\in\fh^\ast_{\oa}$ dominant. One way to prove this is by making use of Proposition 2 in \cite{Frisk}, which implies that
\[\Ann_{U(\fg)}M(w\Lambda)=\Ann_{ U(\fg)}\left(\Ind M_{\oa}(w\circ\Lambda)\right)_{\chi_\Lambda}.\]
The claim then follows from Lemmata \ref{IplusAnn} and \ref{induct} together with Duflo's result. 

Now we consider a singular strongly typical character. Assume that the singular weight $\lambda$ is $W$-maximal in its orbit. As in the proof of Lemma \ref{typeIorbitprop} we can take $\nu$ integral dominant such that $\Lambda=\lambda+\nu$ is strongly typical and dominant. The set of Verma modules $\{M(w\lambda)\,|\,w\in W\}$ is then equal to $\{(M(w\Lambda)\otimes L(\nu)^\ast)_{\chi_\lambda}\,|\,w\in W\}$. The result therefore follows from the regular case and Lemmata \ref{IplusAnn} and \ref{Anntens}.
\end{proof}

\begin{theorem}
\label{Vermaorbit}
Let $\fg$ be a classical Lie superalgebra in the list \eqref{list} with arbitrary Borel subalgebra. The equality 
\[\Ann_{ U(\fg)}M(\mu)= \Ann_{ U(\fg)}M(\lambda)\]
implies there is $w\in W$ such that $\mu=w\cdot\lambda$.
\end{theorem}

\begin{proof}
We use central characters in this proof, which therefore does not include $\fg$ equal to $\mathfrak{p}(n)$ or $\widetilde{\mathfrak{p}}(n)$. However, as proven in Theorem 3.4 or Theorem 3.6 in \cite{MR1943937}, the Jacobson radical $J$ of $U(\fg)$ is included in the annihilator ideal of every Verma module. Thus an inclusion between annihilator ideals of Verma modules in $U(\fg)$ or in $\overline{U}:=U(\fg)/J$ is equivalent. The remainder of the proof can therefore be applied to those Lie superalgebras by replacing $U(\fg)$ by $\overline{U}$.

Assume $\Ann_{ U(\fg)}M(\mu)= \Ann_{ U(\fg)}M(\lambda)$. Denote by $\hat{\mu}$ and $\hat{\lambda}$ the $W$-maximal elements in the orbits (of the $\rho$-shifted action) of $\mu$ and $\lambda$ respectively. If $\hat{\lambda}\not=\hat{\mu}$, without loss of generality we assume that $\hat{\mu} \not\ge \hat{\lambda}$. Then there is an integral dominant weight $\nu$ such that $\Lambda=\hat{\lambda}+\nu$ is strongly typical, see the proof of Lemma \ref{typeIorbitprop}. Since $\left(M(\lambda)\otimes L(\nu)\right)_{\chi_\Lambda}$ has a non-zero filtration by Verma modules while $\left(M(\mu)\otimes L(\nu)\right)_{\chi_\Lambda}$ is trivial, the combination of Lemma \ref{IplusAnn} and Lemma \ref{Anntens} leads to a contradiction. This implies that $\hat{\mu}=\hat{\lambda}$, so $\mu$ and $\lambda$ must be inside the same orbit. 
\end{proof}

\begin{remark}{\rm
It is a remarkable feature that inclusions/equalities between primitive ideals for classical Lie superalgebras follow the structure of the deformed star orbit, while equalities between annihilator ideals of Verma modules follow the structure of the usual $\rho$-shifted Weyl group orbit.}
\end{remark}

\begin{theorem}
\label{AnnVerma}
Let $\fg$ be a classical Lie superalgebra of type I with distinguished system of positive roots. For any $\lambda\in\fh^\ast$ and  $w\in W$, we have the equality
\[\Ann_{ U(\fg)}M(w\cdot\lambda)=\Ann_{ U(\fg)}M(\lambda).\]
Furthermore if $\fg$ is equal to $\mathfrak{sl}(m|n)$ with $m\not=n$ or $\mathfrak{osp}(2|2n)$, the inclusion
\[\Ann_{ U(\fg)}M(\mu)\subseteq \Ann_{ U(\fg)}M(\lambda)\]
implies there is a $w\in W$ such that $\mu=w\cdot\lambda$.
\end{theorem}
\begin{proof}
Lemma \ref{induct} and $M(\mu)\cong \ind^{U(\fg)}_{U(\fg_0+\fg_1)}M_0(\mu)$ imply that \[\Ann_{ U(\fg)}M(w\cdot\lambda)=\Ann_{ U(\fg)}\left( U(\fg)/( U(\fg)\mathfrak{m}^0_{\chi_\lambda^0} U(\fg_1))\right),\]
where we used Duflo's result, which proves the theorem.

If $\fg$ is equal to $\mathfrak{sl}(m|n)$ with $m\not=n$ or $\mathfrak{osp}(2|2n)$ it follows immediately from the structure of the odd roots that $(\Res M(\mu))_{\chi^0_\mu}=M_0(\mu)$. Similarly if $\lambda$ and $\mu$ are in different orbits we have either $(\Res M(\mu))_{\chi_\lambda}=0$ or $(\Res M(\lambda))_{\chi_\mu}=0$. The second claim then follows from Lemma \ref{IplusAnn}.
\end{proof}

\subsection*{Acknowledgment}

KC is a Postdoctoral Fellow of the Research Foundation - Flanders (FWO).
VM is partially supported by the Swedish Research Council.
We thank Ian Musson and Jens Carsten Jantzen for useful discussions.

\end{document}